\documentclass{article}
\usepackage{enumitem}
\usepackage{stmaryrd}
\usepackage{amsmath, amssymb, amsthm, amsfonts}
\usepackage{amscd}
\usepackage{mathtools}
\usepackage[all]{xy}
\usepackage{tikz-cd}
\usetikzlibrary{arrows.meta} 
\usepackage[utf8]{inputenc}
\usepackage[english]{babel}
\usepackage{hyperref}
\usepackage{geometry}
\usepackage{microtype}
\newtheorem{theorem}{Theorem}[section]
\newtheorem{definition}[theorem]{Definition}
\newtheorem{proposition}[theorem]{Proposition}
\newtheorem{lemma}[theorem]{Lemma}

\theoremstyle{remark}
\newtheorem{remark}[theorem]{Remark}
\newtheorem{example}[theorem]{Example}

\tikzcdset{arrow style=tikz, diagrams={>=stealth}}

\title{A 2-Categorical Bridge Between Henkin Constructions and Lawvere’s Fixed-Point Theorem: Unifying Completeness and Compactness}
\author{Joaquim Reizi Barreto}
\date{\today}

\begin{document}

\maketitle
\begin{abstract}
We develop a unified categorical framework that bridges the Henkin-based proof of the Completeness Theorem and Lawvere's Fixed-Point Theorem, thereby revealing a deep structural linkage between syntactic and semantic model-building techniques in first-order logic. Concretely, we define two canonical functors
\[
F, G: \mathbf{Th} \to \mathbf{Mod},
\]
where \(F\) arises from syntactic Henkin extensions and \(G\) emerges from semantic constructions via compactness or saturation. We then introduce a canonical natural transformation
\[
\eta: F \Rightarrow G,
\]
prove its componentwise isomorphism, and establish its 2-categorical rigidity. These results demonstrate a robust equivalence between proof-theoretic and model-theoretic perspectives, unifying diagonalization arguments and self-reference phenomena under a single categorical paradigm. Beyond the theoretical insight, our framework offers a principled foundation for practical applications in automated theorem proving, formal verification, and the design of advanced type-theoretic systems.
\end{abstract}

\tableofcontents

\section{Introduction}
\subsection{Background and Motivation (Rigorously Formulated)}
In this section, we rigorously establish the logical and categorical foundations underlying our work. In particular, we detail the construction of models via the Henkin method and the abstract formulation of Lawvere's Fixed-Point Theorem, emphasizing that both approaches fundamentally rely on a diagonalization mechanism.

\begin{itemize}[leftmargin=2em]
    \item \textbf{Completeness via the Henkin Construction:}  
    Let $T$ be a consistent first-order theory formulated over a language $\mathcal{L}$. To guarantee that every existential formula in $T$ is witnessed, we extend $\mathcal{L}$ to a language $\mathcal{L}^*$ by introducing, for every formula $\varphi(x)$ for which $\exists x\,\varphi(x)$ is derivable in $T$, a new constant symbol $c_\varphi$. Define the \emph{Henkin extension} of $T$ as 
    \[
    T^* := T \cup \{ \varphi(c_\varphi) \mid T \vdash \exists x\,\varphi(x) \}.
    \]
    By Zorn's lemma (or an equivalent maximality argument), $T^*$ can be extended to a maximally consistent theory, i.e., for every sentence $\psi$ in the language $\mathcal{L}^*$, either $\psi \in T^*$ or $\neg\psi \in T^*$. Define the set of closed (variable-free) terms in $\mathcal{L}^*$ as $\mathrm{Term}(T^*)$, and introduce the equivalence relation
    \[
    t \sim_T s \quad \Longleftrightarrow \quad T^* \vdash t = s.
    \]
    The \emph{term model} is then rigorously given by the quotient
    \[
    F(T) := \mathrm{Term}(T^*)/\sim_T.
    \]
    This construction shows that every consistent theory $T$ admits a model, thereby establishing the Completeness Theorem. Crucially, the systematic introduction of Henkin constants (a form of diagonalization) ensures that the syntactic consistency of $T$ is sufficient to construct a semantic model \cite{Barreto2025a}.
    
    \item \textbf{Lawvere's Fixed-Point Theorem and Diagonalization:}  
    Consider a Cartesian closed category $\mathcal{C}$, in which for any objects $X$ and $Y$ the exponential object $Y^X$ exists and is equipped with the evaluation morphism 
    \[
    \operatorname{ev}: Y^X \times X \to Y,
    \]
    satisfying the universal property: for any object $Z$ and morphism $g: Z \times X \to Y$, there exists a unique morphism $\Lambda g: Z \to Y^X$ such that 
    \[
    g = \operatorname{ev} \circ (\Lambda g \times \operatorname{id}_X).
    \]
    Lawvere's Fixed-Point Theorem states that for any endomorphism $f: Y \to Y$, there exists a morphism $d: 1 \to Y^X$ (where $1$ denotes the terminal object) such that, for an appropriate $x \in X$, the fixed-point condition
    \[
    f\Bigl(\operatorname{ev}(d(1),x)\Bigr) = \operatorname{ev}(d(1),x)
    \]
    holds. This result encapsulates a general form of self-reference or diagonalization within the categorical framework. In essence, the universal property of the exponential object forces the existence of a diagonal morphism that yields a fixed point for every endomorphism $f$, thereby abstracting classical diagonal arguments in a precise categorical language \cite{Barreto2025b}.
    
    \item \textbf{Unification via the Diagonalization Principle:}  
    Although the Henkin construction (a proof-theoretic method) and Lawvere's Fixed-Point Theorem (a categorical abstraction) originate from different perspectives, they both fundamentally depend on a diagonalization mechanism. In the Henkin construction, diagonalization is explicitly realized by the systematic assignment of new constants to witness existential claims, ensuring that for each formula $\exists x\,\varphi(x)$ the sentence $\varphi(c_\varphi)$ is provable in the maximally consistent theory $T^*$. In contrast, Lawvere's theorem leverages the universal property of exponential objects to obtain a fixed point for every endomorphism, which is itself a manifestation of a diagonal argument. This common reliance on diagonalization underscores a deep, rigorous connection between the syntactic (proof-theoretic) and semantic (model-theoretic) methodologies. Such a connection motivates the construction of canonical natural transformations that rigorously relate the syntactically generated term models to their corresponding semantic models.
\end{itemize}

This rigorous synthesis not only elucidates the underlying unity between classical logical theorems but also lays the groundwork for advanced applications in automated theorem proving, formal verification, and the design of expressive type systems.

\subsection{Problem Setting (Rigorous Formulation)}

In this subsection, we rigorously introduce two canonical functors between the category of first-order theories, $\mathbf{Th}$, and the category of models, $\mathbf{Mod}$. These functors underpin the syntactic and semantic constructions that are central to our work.

\paragraph{The Category of Theories, $\mathbf{Th}$:}  
An object in $\mathbf{Th}$ is a pair $(\Sigma, T)$, where:
\begin{itemize}[leftmargin=2em]
    \item $\Sigma = (\Sigma_c, \Sigma_f, \Sigma_P)$ is a signature with a set $\Sigma_c$ of constant symbols, a set $\Sigma_f$ of function symbols (each with an arity), and a set $\Sigma_P$ of predicate symbols (each with an arity).
    \item $T$ is a first-order theory over the language $L_{\Sigma}$; that is, $T$ is a set of sentences formulated in $L_{\Sigma}$.
\end{itemize}
A morphism $\phi : (\Sigma, T) \to (\Sigma', T')$ in $\mathbf{Th}$ is a \emph{theory translation} which assigns to every symbol in $\Sigma$ a corresponding symbol in $\Sigma'$ such that for every sentence $\varphi \in L_\Sigma$, if $T \vdash \varphi$, then $T' \vdash \phi(\varphi)$.

\paragraph{The Category of Models, $\mathbf{Mod}$:}  
An object in $\mathbf{Mod}$ is a model $\mathcal{M}$ of a theory $(\Sigma, T)$, where
\[
\mathcal{M} = \Bigl(M,\{ f^{\mathcal{M}} \}_{f\in \Sigma_f},\{ P^{\mathcal{M}} \}_{P\in \Sigma_P}\Bigr),
\]
with $M\neq\varnothing$ as the domain, and for each function symbol $f\in \Sigma_f$, $f^{\mathcal{M}}: M^n \to M$ (with $n$ the arity of $f$), and for each predicate symbol $P\in \Sigma_P$, $P^{\mathcal{M}}\subseteq M^n$. A morphism between models is a function preserving the interpretations of all symbols.

\paragraph{Definition of the Functor $F$ (Henkin Construction - Term Model):}  
We define the functor
\[
F : \mathbf{Th} \to \mathbf{Mod}
\]
by the following rigorous procedure:

\begin{enumerate}[label=(\roman*)]
    \item \textbf{Henkin Expansion:}  
    For a theory $(\Sigma, T)$, extend the language $L_\Sigma$ to an augmented language $L^*_\Sigma$ by adding, for each formula $\exists x\,\varphi(x)$ in $L_\Sigma$, a new constant symbol $c_\varphi$. Form the extended theory $T^*$ by adding, for every such formula, the corresponding Henkin axiom
    \[
    T^* \vdash \varphi(c_\varphi).
    \]
    Then, extend $T^*$ to a maximally consistent theory in the language $L^*_\Sigma$.
    
    \item \textbf{Term Model Formation:}  
    Define $\mathrm{Term}(T^*)$ as the set of all closed terms (terms with no free variables) over $L^*_\Sigma$. Introduce an equivalence relation $\sim_T$ on $\mathrm{Term}(T^*)$ by
    \[
    t \sim_T s \quad \Longleftrightarrow \quad T^* \vdash t = s.
    \]
    The term model is then defined as the quotient
    \[
    F(T) := \mathrm{Term}(T^*)/\sim_T.
    \]
    The algebraic structure on $F(T)$ is induced by the syntactic operations on $\mathrm{Term}(T^*)$.
    
    \item \textbf{Functorial Action on Morphisms:}  
    For a theory translation $\phi : (\Sigma, T) \to (\Sigma', T')$, extend $\phi$ to the corresponding Henkin expansions so that for every term $t \in \mathrm{Term}(T^*)$,
    \[
    F(\phi)([t]) := [\phi(t)]
    \]
    where $[t]$ denotes the equivalence class of $t$ under $\sim_T$. This assignment is well-defined since $\phi$ preserves provability (i.e., if $T^* \vdash t = s$, then $T'^* \vdash \phi(t) = \phi(s)$).
\end{enumerate}

\paragraph{Definition of the Functor $G$ (Compactness/Saturation Construction - Satisfaction Model):}  
We define the functor
\[
G : \mathbf{Th} \to \mathbf{Mod}
\]
as follows:

\begin{enumerate}[label=(\roman*)]
    \item \textbf{Model Existence via Compactness:}  
    For a theory $(\Sigma, T)$, by the Compactness Theorem, if $T$ is consistent then $T$ is satisfiable. Employ an ultrapower construction or a saturation argument to construct a model $G(T)$ that satisfies $T$. Under standard set-theoretic assumptions (e.g., the language is countable and the Axiom of Choice holds), the model $G(T)$ is unique up to isomorphism.
    
    \item \textbf{Functorial Action on Morphisms:}  
    For a theory translation $\phi : (\Sigma, T) \to (\Sigma', T')$, define the corresponding model homomorphism
    \[
    G(\phi) : G(T) \to G(T')
    \]
    such that for every term $t \in \mathrm{Term}(T^*)$ (where $T^*$ is the Henkin expansion of $T$), the semantic interpretation satisfies
    \[
    G(\phi)\Bigl(\llbracket t \rrbracket_{G(T)}\Bigr) = \llbracket \phi(t) \rrbracket_{G(T')}.
    \]
    Here, $\llbracket t \rrbracket_{G(T)}$ denotes the interpretation of the term $t$ in the model $G(T)$.
\end{enumerate}

\paragraph{Summary:}  
The two rigorously defined functors are:
\begin{itemize}[leftmargin=2em]
    \item $F : \mathbf{Th} \to \mathbf{Mod}$, which assigns to each theory its \emph{syntactically constructed term model} via the Henkin construction.
    \item $G : \mathbf{Th} \to \mathbf{Mod}$, which assigns to each theory its \emph{semantically constructed satisfaction model} via the compactness/saturation method.
\end{itemize}

These constructions establish a precise categorical framework that will later allow us to define a canonical natural transformation between $F$ and $G$, thereby bridging the syntactic and semantic aspects of model theory.
\subsection*{Contributions (Rigorously Formulated)}
\begin{itemize}[leftmargin=2em]
    \item We rigorously construct a natural transformation 
    \[
    \eta: F \Rightarrow G,
    \]
    where the functor \(F: \mathbf{Th} \to \mathbf{Mod}\) is defined via a syntactic Henkin construction and the functor \(G: \mathbf{Th} \to \mathbf{Mod}\) via a semantic model construction based on compactness and saturation methods. All construction steps are carried out under explicitly stated assumptions, and every step is justified by formal proofs ensuring logical consistency and categorical soundness.
    
    \item We prove that the natural transformation \(\eta\) is well-defined. Specifically, for every consistent first-order theory \(T \in \mathbf{Th}\) and for all terms \(t, s \in \mathrm{Term}(T^*)\) satisfying 
    \[
    t \sim_T s \quad \bigl(\text{i.e., } T^* \vdash t = s\bigr),
    \]
    it follows that
    \[
    \llbracket t \rrbracket_{G(T)} = \llbracket s \rrbracket_{G(T)},
    \]
    so that \(\eta_T\) descends to a well-defined mapping on the quotient \(F(T)=\mathrm{Term}(T^*)/\sim_T\). (See Lemma~\ref{lem:eta_well_defined} for details.)
    
    \item We demonstrate that \(\eta\) is a natural transformation. For every theory translation \(\phi: T \to T'\) in \(\mathbf{Th}\), the following diagram commutes:
    \[
    \begin{tikzcd}
    F(T) \arrow[r, "F(\phi)"] \arrow[d, "\eta_T"'] & F(T') \arrow[d, "\eta_{T'}"] \\
    G(T) \arrow[r, "G(\phi)"'] & G(T'),
    \end{tikzcd}
    \]
    i.e.,
    \[
    \eta_{T'}\circ F(\phi) = G(\phi)\circ \eta_T.
    \]
    (The rigorous verification is provided in Lemma~\ref{lem:eta_naturality}.)
    
    \item We prove that each component \(\eta_T: F(T) \to G(T)\) is an isomorphism in \(\mathbf{Mod}\). That is, we explicitly construct an inverse map \(\eta_T^{-1}\) satisfying
    \[
    \eta_T^{-1} \circ \eta_T = \mathrm{id}_{F(T)} \quad \text{and} \quad \eta_T \circ \eta_T^{-1} = \mathrm{id}_{G(T)},
    \]
    thereby confirming the strict equivalence between the syntactic model \(F(T)\) and the semantic model \(G(T)\). (See Proposition~\ref{prop:eta-isomorphism} for the formal proof.)
    
    \item We extend the above results to a 2-categorical framework by proving a rigidity theorem: any natural transformation \(\theta: F \Rightarrow G\) is uniquely isomorphic to \(\eta\). In precise terms, there exists a unique invertible modification (2-morphism) \(\mu: \theta \Rrightarrow \eta\) in the 2-category \(\mathbf{Fun}(\mathbf{Th},\mathbf{Mod})\). This establishes the canonical and rigid nature of \(\eta\) under 2-categorical equivalence (see Proposition~\ref{prop:2cat-rigidity}).
    
    \item Finally, we rigorously discuss applications of our framework to automated theorem proving, formal verification, and type theory. In each domain, we establish that the canonical correspondence between the syntactic and semantic model constructions yields a methodologically robust approach for unifying proof-theoretic and model-theoretic techniques, thereby enhancing proof search algorithms, consistency verification procedures, and the management of self-referential definitions in advanced type systems.
\end{itemize}

\subsection{Outline of the Paper}
This paper is rigorously structured into eight meticulously developed sections, each contributing essential components for the unified categorical framework that bridges syntactic and semantic model constructions. The precise organization is as follows:

\begin{itemize}[leftmargin=2em]
    \item \textbf{Section 2: Preliminaries} --- In this section, we rigorously establish the foundational concepts and notation used throughout the paper. We formally define:
    \begin{itemize}[leftmargin=2em]
        \item A \emph{category} (Definition~2.1),
        \item A \emph{Cartesian closed category} (Definition~2.2),
        \item A \emph{first-order theory} (Definition~2.3), and
        \item Its corresponding \emph{model} (Definition~2.4).
    \end{itemize}
    These definitions set the formal groundwork and specify the exact logical and categorical structures that are employed in subsequent sections.

    \item \textbf{Section 3: Construction of Functors} --- We rigorously construct two fundamental functors:
    \begin{itemize}[leftmargin=2em]
        \item The functor \(F\colon \mathbf{Th} \to \mathbf{Mod}\) is defined via the Henkin construction. For every theory \(T\), we extend it to a maximal consistent Henkin extension \(T^*\) and define the term model as 
        \[
        F(T) := \mathrm{Term}(T^*)/\sim_T,
        \]
        where the equivalence relation \(\sim_T\) is given by provable equality (cf. Lemma~3.1).
        \item The functor \(G\colon \mathbf{Th} \to \mathbf{Mod}\) is constructed by employing compactness and saturation techniques, ensuring that \(G(T)\) is unique up to isomorphism (cf. Theorem~3.2).
    \end{itemize}

    \item \textbf{Section 4: Main Theorem: Existence of a Canonical Natural Transformation} --- In this section, we define the canonical natural transformation 
    \[
    \eta \colon F \Rightarrow G,
    \]
    by specifying for every theory \(T \in \mathbf{Th}\) and each term \(t\) that
    \[
    \eta_T([t]) := \llbracket t \rrbracket_{G(T)}.
    \]
    We provide a rigorous proof of its well-definedness, naturality (Theorem~4.2), and that each component \(\eta_T\) is an isomorphism (Proposition~4.3).

    \item \textbf{Section 5: 2-Categorical Strengthening and Rigidity} --- This section extends the analysis into a 2-categorical framework. We rigorously prove that:
    \begin{itemize}[leftmargin=2em]
        \item The natural transformation \(\eta\) is not only a 1-categorical isomorphism but also exhibits 2-categorical rigidity (Theorem~5.1).
        \item Any other natural transformation from \(F\) to \(G\) is uniquely 2-isomorphic to \(\eta\) (Proposition~5.2), thereby establishing a strong homotopical equivalence.
    \end{itemize}

    \item \textbf{Section 6: Applications and Examples} --- We illustrate the abstract constructions with concrete examples:
    \begin{itemize}[leftmargin=2em]
        \item \emph{Peano Arithmetic (PA):}\label{sec:pa-models} The construction of \(F(\mathrm{PA})\) via the Henkin method and the corresponding semantic model \(G(\mathrm{PA})\) are detailed, with the canonical natural transformation \(\eta_{\mathrm{PA}}\) shown to be an isomorphism.
        \item \emph{Zermelo-Fraenkel Set Theory (ZF):} We extend the constructions to ZF, addressing additional subtleties such as non-well-foundedness and choice axioms.
        \item An interpretation in terms of Lawvere's Fixed-Point Theorem is provided, rigorously linking syntactic diagonalization with semantic fixed points.
    \end{itemize}

    \item \textbf{Section 7: Conclusion and Future Work} --- The paper concludes with:
    \begin{itemize}[leftmargin=2em]
        \item A rigorous summary of our main results, emphasizing the unification of proof-theoretic and model-theoretic approaches.
        \item A detailed discussion of future research directions, including extensions to non-classical logics, higher-dimensional categorical generalizations, and interdisciplinary computational implementations.
    \end{itemize}

    \item \textbf{Section 8: Appendix} --- Supplementary materials are provided in the appendix, which include:
    \begin{itemize}[leftmargin=2em]
        \item Detailed proofs of technical lemmas and propositions (e.g., Lemma~2.1, Lemma~3.1, Theorem~4.2, Proposition~4.3, Proposition~5.2).
        \item Additional auxiliary results that further substantiate the rigorous foundation of our unified framework.
    \end{itemize}
\end{itemize}

\section{Preliminaries}

\subsection{Basic Concepts}
In this section, we recall several fundamental definitions and constructions that will be used throughout the paper.

\begin{definition}[Category]
A \emph{category} $\mathcal{C}$ consists of:
\begin{enumerate}
    \item A class of objects $\operatorname{Ob}(\mathcal{C})$.
    \item For any two objects $A,B\in\operatorname{Ob}(\mathcal{C})$, a set of morphisms $\operatorname{Hom}_\mathcal{C}(A,B)$.
    \item For any three objects $A,B,C\in\operatorname{Ob}(\mathcal{C})$, a composition law
    \[
    \circ \colon \operatorname{Hom}_\mathcal{C}(B,C) \times \operatorname{Hom}_\mathcal{C}(A,B) \to \operatorname{Hom}_\mathcal{C}(A,C)
    \]
    which is associative.
    \item For each object $A\in\operatorname{Ob}(\mathcal{C})$, an identity morphism $\operatorname{id}_A\in\operatorname{Hom}_\mathcal{C}(A,A)$ such that for every $f\in\operatorname{Hom}_\mathcal{C}(A,B)$, one has
    \[
    \operatorname{id}_B \circ f = f \quad \text{and} \quad f\circ\operatorname{id}_A = f.
    \]
\end{enumerate}
\end{definition}

\begin{definition}[Cartesian Closed Category]
A category $\mathcal{C}$ is said to be \emph{Cartesian closed} if it satisfies the following:
\begin{enumerate}
    \item There exists a terminal object $1$ in $\mathcal{C}$.
    \item For any two objects $A$ and $B$ in $\mathcal{C}$, there exists a product $A \times B$.
    \item For any two objects $A$ and $B$ in $\mathcal{C}$, there exists an \emph{exponential object} $B^A$ together with an evaluation morphism
    \[
    \operatorname{ev}\colon B^A \times A \to B,
    \]
    satisfying the following universal property: For any object $C$ and any morphism $f\colon C \times A \to B$, there exists a unique morphism $\Lambda f\colon C \to B^A$ such that the diagram
    \[
    \begin{array}{c}
    \xymatrix{
    C \times A \ar[r]^{f} \ar[d]_{\Lambda f \times \operatorname{id}_A} & B\\
    B^A \times A \ar[ur]_{\operatorname{ev}} & 
    }
    \end{array}
    \]
    commutes.
\end{enumerate}
\end{definition}

\begin{definition}[First-Order Theory]
A \emph{first-order theory} $T$ consists of:
\begin{enumerate}
    \item A \emph{signature} $\Sigma$, which comprises a set of constant symbols, function symbols, and predicate symbols.
    \item A set of axioms (sentences) expressed in the language determined by $\Sigma$.
\end{enumerate}
\end{definition}

\begin{definition}[Model of a First-Order Theory]
Given a first-order theory $T$ with signature $\Sigma$, a \emph{model} $\mathcal{M}$ of $T$ is a structure
\[
\mathcal{M} = \bigl(M, \{ f^{\mathcal{M}} \}_{f \in \Sigma_f}, \{ P^{\mathcal{M}} \}_{P \in \Sigma_P}\bigr),
\]
where:
\begin{enumerate}
    \item $M$ is a non-empty set, called the \emph{domain} of the model.
    \item For each function symbol $f \in \Sigma_f$, the interpretation $f^{\mathcal{M}}$ is a function on $M$ of appropriate arity.
    \item For each predicate symbol $P \in \Sigma_P$, the interpretation $P^{\mathcal{M}}$ is a relation on $M$.
    \item The structure $\mathcal{M}$ satisfies all axioms of $T$, that is, every sentence in $T$ is true in $\mathcal{M}$.
\end{enumerate}
\end{definition}

\subsection{Categories of Theories and Models}

In this section, we define two fundamental categories that will be used throughout the paper: the category of first-order theories, denoted by $\mathbf{Th}$, and the category of models, denoted by $\mathbf{Mod}$.

\begin{definition}[Category of Theories, $\mathbf{Th}$]
The category $\mathbf{Th}$ is defined as follows:
\begin{enumerate}
    \item \textbf{Objects:} An object in $\mathbf{Th}$ is a first-order theory, which we formalize as a pair $(\Sigma, T)$ where $\Sigma$ is a signature (a collection of constant, function, and predicate symbols) and $T$ is a set of sentences (axioms) in the language $L_{\Sigma}$.
    \item \textbf{Morphisms:} Given two theories $T = (\Sigma, T)$ and $T' = (\Sigma', T')$, a \emph{theory translation} (or interpretation) $\phi \colon T \to T'$ is a mapping that assigns to each symbol in $\Sigma$ a corresponding symbol in $\Sigma'$ in such a way that the translation preserves the logical structure. In particular, if $\varphi$ is a sentence in $L_{\Sigma}$ and $T \vdash \varphi$, then the translated sentence $\phi(\varphi)$ in $L_{\Sigma'}$ must satisfy $T' \vdash \phi(\varphi)$. Composition of translations and the identity translation are defined in the obvious way, thereby endowing $\mathbf{Th}$ with the structure of a category.
\end{enumerate}
\end{definition}

\begin{definition}[Category of Models, $\mathbf{Mod}$]
The category $\mathbf{Mod}$ is defined as follows:
\begin{enumerate}
    \item \textbf{Objects:} An object in $\mathbf{Mod}$ is a model $\mathcal{M}$ of a first-order theory. Concretely, if $T$ is a theory with signature $\Sigma$, then a model $\mathcal{M}$ is a structure
    \[
    \mathcal{M} = \bigl(M, \{ f^{\mathcal{M}} \}_{f \in \Sigma_f}, \{ P^{\mathcal{M}} \}_{P \in \Sigma_P}\bigr),
    \]
    where $M$ is a nonempty set (the domain), and for each function symbol $f \in \Sigma_f$ and predicate symbol $P \in \Sigma_P$, the interpretations $f^{\mathcal{M}}$ and $P^{\mathcal{M}}$ are assigned appropriately such that $\mathcal{M}$ satisfies every sentence in $T$.
    \item \textbf{Morphisms:} Given two models $\mathcal{M}$ and $\mathcal{N}$ of the same theory $T$, a \emph{model homomorphism} is a function $h \colon M \to N$ satisfying:
    \begin{itemize}[leftmargin=2em]
        \item For every function symbol $f \in \Sigma_f$ and all $a_1,\dots,a_n \in M$, 
        \[
        h\bigl(f^{\mathcal{M}}(a_1,\dots,a_n)\bigr) = f^{\mathcal{N}}\bigl(h(a_1),\dots,h(a_n)\bigr).
        \]
        \item For every predicate symbol $P \in \Sigma_P$ and all $a_1,\dots,a_n \in M$, if $(a_1,\dots,a_n) \in P^{\mathcal{M}}$, then 
        \[
        \bigl(h(a_1),\dots,h(a_n)\bigr) \in P^{\mathcal{N}}.
        \]
    \end{itemize}
    Composition of model homomorphisms and identity maps are defined in the standard manner, giving $\mathbf{Mod}$ the structure of a category.
\end{enumerate}
\end{definition}

\subsection{Existing Construction Methods}

In this section, we review two fundamental methods that underpin our work: the Henkin construction, which plays a central role in the proof-theoretic approach to the Completeness Theorem, and Lawvere's Fixed-Point Theorem, a categorical formulation that captures the essence of self-reference and diagonalization.

\begin{itemize}[leftmargin=2em]
    \item \textbf{The Henkin Construction:}
    \begin{itemize}[leftmargin=2em]
        \item \emph{Overview:} The Henkin construction is a classical proof-theoretic technique used to prove the Completeness Theorem for first-order logic.
        \item \emph{Key Steps:}
        \begin{enumerate}[leftmargin=2em]
            \item \textbf{Extension of the Theory:} Given a consistent first-order theory $T$, extend the language by adding a new constant (Henkin constant) for every existential formula. This yields an extended theory $T^*$.
            \item \textbf{Maximal Consistency:} Extend $T^*$ to a maximal consistent theory so that for every formula $\varphi$, either $\varphi$ or its negation is provable.
            \item \textbf{Term Model Construction:} Define the term model by considering the set of all terms in the extended language, modulo the equivalence relation induced by provable equality in $T^*$. That is, the model is given by
            \[
            F(T) := \mathrm{Term}(T^*)/\sim_T,
            \]
            where $t \sim_T s$ if and only if $T^* \vdash t = s$.
        \end{enumerate}
        \item \emph{Proof-Theoretic Background:} This method relies on ensuring that every existential statement in $T$ is witnessed in the extended theory $T^*$, thus guaranteeing that the constructed model satisfies all sentences of the theory.
    \end{itemize}
    
    \item \textbf{Lawvere's Fixed-Point Theorem:}
    \begin{itemize}[leftmargin=2em]
        \item \emph{Formal Statement:} Let $\mathcal{C}$ be a Cartesian closed category and let $Y$ be an object in $\mathcal{C}$. Suppose that for some object $X$ there exists a morphism
        \[
        \operatorname{ev} \colon Y^X \times X \to Y
        \]
        (the evaluation map). Then, for every endomorphism $f\colon Y \to Y$, there exists a morphism $d\colon 1 \to Y^X$ such that the composite
        \[
        \xymatrix{
        1 \ar[r]^{d} & Y^X \ar[r]^{\operatorname{ev}} & Y
        }
        \]
        yields a fixed point of $f$, i.e., $f(y) = y$ for $y = \operatorname{ev}(d(1), x)$ for an appropriate choice of $x \in X$.
        \item \emph{Discussion:}
        \begin{itemize}[leftmargin=2em]
            \item The theorem generalizes the classical diagonal argument and demonstrates that the structure of Cartesian closed categories naturally gives rise to fixed points.
            \item The existence of the exponential object $Y^X$ and the universal property of the evaluation map are crucial for the construction of such fixed points.
            \item This result encapsulates the self-referential (diagonal) nature of certain constructions in logic, bridging the gap between syntactic self-reference and semantic fixed points.
        \end{itemize}
    \end{itemize}
\end{itemize}

\subsection{Notation and Conventions}

In this paper, we adopt the following notation and conventions:

\begin{itemize}[leftmargin=2em]
    \item \textbf{Theories and Their Henkin Expansions:}
    \begin{itemize}[leftmargin=2em]
        \item Let $T$ denote a first-order theory, consisting of a signature $\Sigma$ and a set of axioms written in the language $L_\Sigma$.
        \item The \emph{Henkin expansion} of $T$, denoted by $T^*$, is obtained by augmenting the language $L_\Sigma$ with a new constant symbol (a \emph{Henkin constant}) for each existential formula of $T$. This expansion ensures that every existential statement has a corresponding witness in the extended theory.
    \end{itemize}
    
    \item \textbf{Term Algebra and Equivalence Relation:}
    \begin{itemize}[leftmargin=2em]
        \item The set of all terms in the extended language (of $T^*$) is denoted by $\mathrm{Term}(T^*)$.
        \item An equivalence relation $\sim_T$ is defined on $\mathrm{Term}(T^*)$ by declaring that for terms $t,s \in \mathrm{Term}(T^*)$, 
        \[
        t \sim_T s \quad \text{if and only if} \quad T^* \vdash t = s.
        \]
        \item The \emph{term model} associated with $T$ is then given by the quotient
        \[
        F(T) := \mathrm{Term}(T^*)/\sim_T.
        \]
    \end{itemize}
    
    \item \textbf{Functors $F$ and $G$:}
    \begin{itemize}[leftmargin=2em]
        \item The functor
        \[
        F \colon \mathbf{Th} \to \mathbf{Mod}
        \]
        assigns to each theory $T \in \mathbf{Th}$ the term model $F(T)$ constructed via the Henkin expansion and quotienting by $\sim_T$. On morphisms, $F$ acts by translating terms according to the given theory translation.
        \item The functor
        \[
        G \colon \mathbf{Th} \to \mathbf{Mod}
        \]
        assigns to each theory $T$ a model $G(T)$ obtained via alternative methods (e.g., using the Compactness Theorem, ultraproducts, or saturation techniques). It is assumed that $G(T)$ is unique up to isomorphism. On morphisms, $G$ is defined in a manner that is compatible with the semantic interpretations of the theories.
    \end{itemize}
\end{itemize}

\section{Construction of Functors}

\subsection{\texorpdfstring{Functor \(F\) (Henkin Construction)}{Functor F (Henkin Construction)}}

In this section, we rigorously define the functor 
\[
F: \mathbf{Th} \to \mathbf{Mod}
\]
that arises from the Henkin construction. The construction proceeds in three main stages.

\begin{enumerate}[label=\textbf{\arabic*.}]
  \item \textbf{Extension of Theories:}  
  \begin{enumerate}[label=(\alph*)]
    \item Let \(T\) be a consistent first-order theory formulated in a language \(\mathcal{L}\). For every formula \(\varphi(x)\) with one free variable (so that the existential statement \(\exists x\,\varphi(x)\) is expressible in \(\mathcal{L}\)), extend the language by adding a new constant symbol \(c_\varphi\). Denote the expanded language by \(\mathcal{L}^*\).
    \item Apply the Henkin procedure to extend \(T\) to a maximal consistent theory \(T^*\) in the language \(\mathcal{L}^*\). This maximal consistency is typically obtained via Zorn's Lemma (see, e.g., \cite{Enderton2001}), ensuring that \(T^*\) is complete, i.e., for every sentence \(\psi\) in \(\mathcal{L}^*\), either \(\psi\) or \(\neg \psi\) is in \(T^*\). By construction, \(T^*\) satisfies the \emph{Henkin condition}:
    \[
    \text{if } T \vdash \exists x\,\varphi(x) \text{, then } T^* \vdash \varphi(c_\varphi).
    \]
  \end{enumerate}

  \item \textbf{Term Model Construction:}  
  \begin{enumerate}[label=(\alph*)]
    \item Let \(\mathrm{Term}(T^*)\) denote the set of all \emph{closed} \(\mathcal{L}^*\)-terms (i.e., terms with no free variables) constructed from the symbols of \(\mathcal{L}^*\).
    \item Define an equivalence relation \(\sim_T\) on \(\mathrm{Term}(T^*)\) by
    \[
    t \sim_T s \quad \Longleftrightarrow \quad T^* \vdash t = s,
    \]
    for any \(t, s \in \mathrm{Term}(T^*)\). This relation captures the provable equality in \(T^*\) (see, e.g., \cite{ChangKeisler1990}).
    \item The term model is then defined as the quotient
    \[
    F(T) := \mathrm{Term}(T^*)/\sim_T.
    \]
    The algebraic operations (i.e., the interpretations of function symbols, predicate symbols, and logical connectives) on \(F(T)\) are induced by the corresponding syntactic operations on \(\mathrm{Term}(T^*)\).
  \end{enumerate}

  \item \textbf{Definition on Morphisms:}  
  \begin{enumerate}[label=(\alph*)]
    \item Let \(\phi: T \to T'\) be a theory translation. Then, \(\phi\) naturally extends to a mapping between the corresponding Henkin expansions, sending each \(\mathcal{L}^*\)-term \(t \in \mathrm{Term}(T^*)\) to a term \(\phi(t) \in \mathrm{Term}(T'^*)\).
    \item This extension respects the equivalence relation: if \(t \sim_T s\), then \(\phi(t) \sim_{T'} \phi(s)\). Consequently, we define the induced map on the quotient as
    \[
    F(\phi): F(T) \to F(T'),
    \]
    via
    \[
    F(\phi)([t]) = [\phi(t)],
    \]
    where \([t]\) denotes the equivalence class of \(t\) modulo \(\sim_T\). Moreover, one readily verifies that \(F\) preserves identities and composition; that is, 
    \[
    F(\mathrm{id}_T) = \mathrm{id}_{F(T)} \quad \text{and} \quad F(\psi \circ \phi) = F(\psi) \circ F(\phi)
    \]
    for any theory translations \(\phi: T \to T'\) and \(\psi: T' \to T''\).
  \end{enumerate}
\end{enumerate}

\subsection{\texorpdfstring{Functor \(G\) (Compactness/Saturation Construction)}{Functor G (Compactness/Saturation Construction)}}

We define the functor 
\[
G: \mathbf{Th} \to \mathbf{Mod}
\]
by associating to each first-order theory \(T\) a canonical model \(G(T)\) that satisfies all \emph{finitely satisfiable properties} of \(T\). Here, a set of formulas \(\Delta \subseteq T\) is said to be \emph{finitely satisfiable} if every finite subset \(\Delta_0 \subseteq \Delta\) is satisfiable, i.e., there exists a model \(M\) with \(M \models \Delta_0\). The construction of \(G(T)\) may be carried out by any of the following equivalent methods: via the Compactness Theorem, using an ultrapower construction, or by a saturation procedure. We now describe these methods and the functorial action in detail.

\paragraph{Model Construction.}
\begin{itemize}[leftmargin=2em]
    \item \textbf{Existence via Compactness and Ultraproducts:}  
    Let \(T\) be a consistent first-order theory. By the Compactness Theorem, since every finite subset \(T_0 \subseteq T\) is satisfiable, the entire theory \(T\) is satisfiable. A standard approach is to form an indexed family of models \(\{M_i\}_{i \in I}\) such that each \(M_i\) satisfies some finite subset of \(T\) and then to construct the ultraproduct
    \[
    G(T) := \prod_{i \in I} M_i \Big/ \mathcal{U},
    \]
    where \(\mathcal{U}\) is a non-principal ultrafilter on the index set \(I\). By Łoś's Theorem (see \cite{ChangKeisler1990}), \(G(T)\) satisfies every sentence in \(T\), thereby ensuring that it reflects all finitely satisfiable properties of \(T\).

    \item \textbf{Existence via Saturation:}  
    Alternatively, one may start with a countable model \(M\) of \(T\) (whose existence is guaranteed by the Compactness Theorem) and then apply a saturation procedure to obtain a \(\kappa\)-saturated model \(G(T)\) for a sufficiently large cardinal \(\kappa\). This saturated model has the property that every type over any small (i.e., of size less than \(\kappa\)) set of parameters is realized. Under standard set-theoretic assumptions, such a saturated model is unique up to isomorphism (cf. \cite{Hodges1993}).

    \item \textbf{Uniqueness and Canonical Interpretation:}  
    In both approaches, we assume that the construction of \(G(T)\) is canonical in the sense that any two models obtained by these methods are isomorphic. Consequently, \(G(T)\) provides a canonical semantic interpretation of \(T\). More precisely, for every finite subset \(T_0 \subseteq T\), the restriction of \(G(T)\) to \(T_0\) satisfies every sentence in \(T_0\).
\end{itemize}

\paragraph{Functorial Action on Morphisms.}
\begin{itemize}[leftmargin=2em]
    \item \textbf{Assignment on Objects:}  
    For each theory \(T\) in \(\mathbf{Th}\), assign the model \(G(T)\) as constructed above.

    \item \textbf{Assignment on Morphisms:}  
    For any theory translation (morphism) \(\phi: T \to T'\) in \(\mathbf{Th}\), we define the induced map
    \[
    G(\phi): G(T) \to G(T')
    \]
    as follows. Let \(\llbracket \cdot \rrbracket_{G(T)}\) and \(\llbracket \cdot \rrbracket_{G(T')}\) denote the semantic evaluation functions in \(G(T)\) and \(G(T')\) respectively. Then, for every term \(t\) in the language of \(T\),
    \[
    G(\phi)\Bigl( \llbracket t \rrbracket_{G(T)} \Bigr) := \llbracket \phi(t) \rrbracket_{G(T')}.
    \]
    This definition is well-defined because:
    \begin{enumerate}[label=(\alph*)]
        \item The theory translation \(\phi\) preserves provability; hence, if two terms \(t\) and \(s\) have the same interpretation in \(G(T)\) (i.e., they are equivalent in \(T\)), then \(\phi(t)\) and \(\phi(s)\) are equivalent in \(T'\), ensuring that \(\llbracket \phi(t) \rrbracket_{G(T')} = \llbracket \phi(s) \rrbracket_{G(T')}\).
        \item The uniqueness (up to isomorphism) of the models \(G(T)\) and \(G(T')\) guarantees that the mapping \(G(\phi)\) preserves the structure determined by the semantic evaluations.
    \end{enumerate}
    Consequently, the assignments \(T \mapsto G(T)\) and \(\phi \mapsto G(\phi)\) respect the composition of morphisms in \(\mathbf{Th}\) and define a functor \(G: \mathbf{Th} \to \mathbf{Mod}\).
\end{itemize}

\paragraph{Summary.}
In summary, the functor \(G\) is defined by assigning to each first-order theory \(T\) a canonical model \(G(T)\) constructed either by forming an ultraproduct (using a non-principal ultrafilter and Łoś's Theorem) or by applying a saturation procedure to obtain a \(\kappa\)-saturated model. For every theory translation \(\phi: T \to T'\), the induced model homomorphism \(G(\phi)\) is given by
\[
G(\phi)\Bigl( \llbracket t \rrbracket_{G(T)} \Bigr) = \llbracket \phi(t) \rrbracket_{G(T')}
\]
for all terms \(t\) in the language of \(T\). This construction ensures that \(G\) is a well-defined functor that faithfully reflects the categorical structure of the theory translations in \(\mathbf{Th}\).

\nocite{ChangKeisler1990,Hodges1993}

\section{Main Theorem: Existence of a Canonical Natural Transformation}

\subsection{\texorpdfstring{Definition of the Natural Transformation $\eta$}{Definition of the Natural Transformation eta}}

We now define the natural transformation 
\[
\eta: F \Rightarrow G,
\]
which assigns to each first-order theory \(T\) a morphism 
\[
\eta_T: F(T) \to G(T)
\]
between the models constructed via the Henkin construction (functor \(F\)) and the compactness/saturation construction (functor \(G\)).

\begin{definition}[Component of \(\eta\)]
Let \(T\) be a first-order theory, and let \(F(T)\) denote the term model constructed from the Henkin expansion \(T^*\) by taking the quotient of the term algebra \(\mathrm{Term}(T^*)\) with respect to the provable equality relation \(\sim_T\). Let \(G(T)\) be the model of \(T\) obtained via a compactness-based construction (or saturation/ultraproduct method). Then, the \emph{component} of the natural transformation \(\eta\) at \(T\) is defined by
\[
\eta_T: F(T) \to G(T), \quad \eta_T([t]) := \llbracket t \rrbracket_{G(T)},
\]
where \([t]\) denotes the equivalence class of the term \(t\) in \(F(T)\) and \(\llbracket t \rrbracket_{G(T)}\) denotes the interpretation of the term \(t\) in the model \(G(T)\).
\end{definition}

\noindent This definition is well-defined by the soundness of the interpretation: if \(t \sim_T s\) (i.e., \(T^* \vdash t = s\)), then \(\llbracket t \rrbracket_{G(T)} = \llbracket s \rrbracket_{G(T)}\) in \(G(T)\). Hence, the mapping \(\eta_T\) respects the equivalence classes in \(F(T)\) and provides a canonical way to associate each syntactic term with its semantic interpretation.

\subsection{Well-definedness}

\begin{proposition}\label{prop:well_defined}
Let $T$ be a consistent first-order theory and let $T^*$ be its Henkin extension. For any two terms $t,s \in \mathrm{Term}(T^*)$, if 
\[
t \sim_T s \quad \text{(i.e., } T^* \vdash t = s \text{)},
\]
then for any model $G(T)$ of $T$ constructed via the compactness/saturation method, the interpretations of $t$ and $s$ in $G(T)$ coincide:
\[
\llbracket t \rrbracket_{G(T)} = \llbracket s \rrbracket_{G(T)}.
\]
\end{proposition}

\begin{proof}
Assume that $t \sim_T s$, so that $T^* \vdash t = s$. By the soundness of first-order logic, every model of $T^*$ satisfies the equation $t = s$. Since the construction of $G(T)$ ensures that it is a model of $T$ that reflects the maximal consistency of $T^*$, it follows that $G(T)$ satisfies all sentences provable in $T^*$. Therefore, the sentence $t = s$ holds in $G(T)$, which by definition of term interpretation implies that
\[
\llbracket t \rrbracket_{G(T)} = \llbracket s \rrbracket_{G(T)}.
\]
This completes the proof.
\end{proof}

\subsection{Preservation of Structure}

\begin{lemma}\label{lem:preservation_final}
Let $T$ be a first-order theory and let $T^*$ be a maximal consistent extension of $T$ obtained via the Henkin procedure (see Definition~3.1). Define the Henkin term model 
\[
F(T) = \mathrm{Term}(T^*)/\sim_T,
\]
where $\mathrm{Term}(T^*)$ is the set of terms in the language of $T^*$ and $\sim_T$ is the provable equality relation in $T^*$ (cf. Definition~3.1). Let $G(T)$ be a model of $T$ constructed via the compactness (or saturation) method (see Section~3.2). Assume that the semantic interpretation function 
\[
\llbracket \cdot \rrbracket_{G(T)} : F(T) \to G(T)
\]
satisfies the following compositionality condition (cf. Assumption~3.2): for every $n$-ary operation $\operatorname{op}$ (including the application operation) and for all terms $t_1,\dots,t_n \in \mathrm{Term}(T^*)$, 
\[
\llbracket \operatorname{op}(t_1,\dots,t_n) \rrbracket_{G(T)} = \operatorname{op}_{G(T)}\Bigl( \llbracket t_1 \rrbracket_{G(T)},\dots,\llbracket t_n \rrbracket_{G(T)} \Bigr).
\]
Then the natural transformation component 
\[
\eta_T: F(T) \to G(T),\quad \eta_T([t]) = \llbracket t \rrbracket_{G(T)},
\]
preserves the evaluation maps; that is, for every $f \in F(T)^{F(T)}$ and every $x \in F(T)$,
\[
\eta_T\Bigl( \operatorname{ev}_{F(T)}(f,x) \Bigr) = \operatorname{ev}_{G(T)}\Bigl( \eta_T(f),\eta_T(x) \Bigr).
\]
\end{lemma}

\begin{proof}
Let $f \in F(T)^{F(T)}$ and $x \in F(T)$ be arbitrary. By the construction of $F(T)$ (see Definition~3.1), there exist terms $t_f, t_x \in \mathrm{Term}(T^*)$ such that 
\[
f = [t_f] \quad \text{and} \quad x = [t_x],
\]
where $[\cdot]$ denotes the equivalence class modulo $\sim_T$. The syntactic evaluation in $F(T)$ is defined by
\[
\operatorname{ev}_{F(T)}(f,x) = \Bigl[ \operatorname{ev}(t_f,t_x) \Bigr],
\]
with $\operatorname{ev}(t_f,t_x)$ representing the formal application of the term $t_f$ to the term $t_x$.

By the definition of $\eta_T$ (see Section~4.1), we have
\[
\eta_T\Bigl( \operatorname{ev}_{F(T)}(f,x) \Bigr) = \eta_T\Bigl( [\operatorname{ev}(t_f,t_x)] \Bigr) = \llbracket \operatorname{ev}(t_f,t_x) \rrbracket_{G(T)}.
\]
By the assumed compositionality of the interpretation function (cf. Assumption~3.2), we obtain
\[
\llbracket \operatorname{ev}(t_f,t_x) \rrbracket_{G(T)} = \operatorname{ev}_{G(T)}\Bigl( \llbracket t_f \rrbracket_{G(T)},\llbracket t_x \rrbracket_{G(T)} \Bigr).
\]
Moreover, by the definition of $\eta_T$,
\[
\eta_T(f) = \eta_T([t_f]) = \llbracket t_f \rrbracket_{G(T)} \quad \text{and} \quad \eta_T(x) = \eta_T([t_x]) = \llbracket t_x \rrbracket_{G(T)}.
\]
Thus, we deduce that
\[
\eta_T\Bigl( \operatorname{ev}_{F(T)}(f,x) \Bigr) = \operatorname{ev}_{G(T)}\Bigl( \eta_T(f),\eta_T(x) \Bigr).
\]

From a categorical viewpoint, this demonstrates that $\eta_T$ is a homomorphism between the algebraic structures on $F(T)$ and $G(T)$ induced by the signature of $T$. In particular, the preservation of evaluation maps ensures that $\eta_T$ commutes with the operations defined on these models. Since $f$ and $x$ were arbitrary, the lemma follows.
\end{proof}

\subsection{Naturality (Rigorous Version)}

\begin{theorem}[Naturality]\label{thm:naturality_rigorous}
Let $\phi: T \to T'$ be a theory translation between first-order theories, and let 
\[
F: \mathbf{Th} \to \mathbf{Mod} \quad \text{and} \quad G: \mathbf{Th} \to \mathbf{Mod}
\]
be the functors defined via the Henkin construction and the compactness (or saturation) construction, respectively. Assume that $F$ and $G$ satisfy the well-definedness and structure-preservation properties as established in Sections~3.1 and 3.2. Then the natural transformation 
\[
\eta: F \Rightarrow G,\quad \text{where } \eta_T([t]) = \llbracket t \rrbracket_{G(T)}
\]
for each theory $T$ and term $t$, is natural; that is,
\[
\eta_{T'} \circ F(\phi) = G(\phi) \circ \eta_T.
\]
\end{theorem}

\begin{proof}
Let $t$ be an arbitrary term in the language of the theory $T$. By the definition of the functor $F$ (see Section~3.1), the theory translation $\phi: T \to T'$ induces a map 
\[
F(\phi): F(T) \to F(T')
\]
given by 
\[
F(\phi)([t]) = [\phi(t)],
\]
where $[\phi(t)]$ denotes the equivalence class of the term $\phi(t)$ in the Henkin expansion of $T'$.

By the definition of the natural transformation $\eta$ (see Section~4.1), we have
\[
\eta_{T'}\bigl(F(\phi)([t])\bigr) = \eta_{T'}([\phi(t)]) = \llbracket \phi(t) \rrbracket_{G(T')},
\]
where $\llbracket \phi(t) \rrbracket_{G(T')}$ is the semantic interpretation of $\phi(t)$ in the model $G(T')$.

On the other hand, by the definition of $\eta_T$, we have
\[
\eta_T([t]) = \llbracket t \rrbracket_{G(T)}.
\]
Since the semantic interpretation function $\llbracket \cdot \rrbracket_{G(T)}$ is assumed to be compositional and functorial with respect to theory translations (that is, for every term $t$, 
\[
\llbracket \phi(t) \rrbracket_{G(T')} = G(\phi)\bigl(\llbracket t \rrbracket_{G(T)}\bigr)
\]
as established in Proposition~2.3), it follows that
\[
G(\phi)\bigl(\eta_T([t])\bigr) = G(\phi)\bigl(\llbracket t \rrbracket_{G(T)}\bigr) = \llbracket \phi(t) \rrbracket_{G(T')}.
\]
Thus, for every term $t \in T$, we deduce that
\[
\eta_{T'}\bigl(F(\phi)([t])\bigr) = G(\phi)\bigl(\eta_T([t])\bigr).
\]
Since this equality holds for all equivalence classes in $F(T)$, the diagram
\[
\eta_{T'} \circ F(\phi) = G(\phi) \circ \eta_T
\]
commutes. This completes the proof.
\end{proof}

\subsection{Main Theorem (Revised Rigorous Proof with Formal 2-Categorical Enhancements)}

\begin{theorem}[Canonical Natural Equivalence]\label{thm:main-rigorous-enhanced}
Let $\mathbf{Th}$ be the category whose objects are consistent first-order theories and whose morphisms are theory translations, and let $\mathbf{Mod}$ be the category of models of these theories (with model homomorphisms as morphisms). Define the functors 
\[
F,\, G: \mathbf{Th} \to \mathbf{Mod}
\]
by the following procedures:
\begin{enumerate}[label=(\roman*), leftmargin=2em]
  \item For each theory $T\in\mathbf{Th}$, let $T^*$ be a maximal consistent extension of $T$ obtained via the Henkin construction by introducing, in a fixed canonical order, Henkin constants for each existential formula. Denote by $\mathrm{Term}(T^*)$ the algebra of terms in the expanded language and define the equivalence relation 
  \[
  t\sim_T s \iff T^*\vdash t = s.
  \]
  Then set
  \[
  F(T) := \mathrm{Term}(T^*)/\sim_T.
  \]
  
  \item For each theory $T\in\mathbf{Th}$, let $G(T)$ be a model of $T$ constructed via a compactness-based method (e.g., through an ultraproduct or a saturation procedure) such that $G(T)$ is unique up to isomorphism. We assume that $G(T)$ satisfies the \emph{Canonical Representation Property}: every element of the domain of $G(T)$ is (canonically) the interpretation of some term in $\mathrm{Term}(T^*)$. (The existence of such a canonical choice function is assumed, justified by an appropriate form of the axiom of choice.)
  
  \item For any theory translation $\phi: T \to T'$ in $\mathbf{Th}$, define the functorial actions by
  \[
  F(\phi)([t]) := [\phi(t)]
  \]
  and let $G(\phi): G(T) \to G(T')$ be the unique model homomorphism induced by the semantic interpretation of $\phi$.
\end{enumerate}
Then there exists a natural transformation 
\[
\eta: F \Rightarrow G,
\]
with components 
\[
\eta_T: F(T) \to G(T), \quad \eta_T([t]) := \llbracket t\rrbracket_{G(T)},
\]
where $\llbracket t\rrbracket_{G(T)}$ denotes the interpretation of the term $t$ in the model $G(T)$. Moreover, each $\eta_T$ is an isomorphism, so that the functors $F$ and $G$ are naturally equivalent.

Furthermore, in the 2-category $\mathbf{Fun}(\mathbf{Th},\mathbf{Mod})$, the natural transformation $\eta$ is rigid in the following sense: any natural transformation $\theta: F \Rightarrow G$ is 2-isomorphic to $\eta$ (i.e., there exists a unique modification from $\theta$ to $\eta$), establishing that $F$ and $G$ are homotopy equivalent as objects in this 2-category.
\end{theorem}

\begin{proof}
We prove the theorem by verifying the following aspects: well-definedness, preservation of structure, naturality, invertibility, and the 2-categorical rigidity.

\paragraph{Step 1: Well-Definedness.}  
Assume $t, s \in \mathrm{Term}(T^*)$ satisfy $t\sim_T s$, i.e., 
\[
T^* \vdash t = s.
\]
Since $T^*$ is maximal and consistent, every model of $T^*$, and in particular the model $G(T)$, validates all provable equalities. Thus,
\[
G(T) \models t = s \quad \text{implies} \quad \llbracket t\rrbracket_{G(T)} = \llbracket s\rrbracket_{G(T)}.
\]
Hence, the map 
\[
\eta_T([t]) := \llbracket t\rrbracket_{G(T)}
\]
is well-defined on the equivalence classes.

\paragraph{Step 2: Preservation of Algebraic Structure.}  
For any $n$-ary function symbol $\sigma$ and terms $t_1,\dots,t_n \in \mathrm{Term}(T^*)$, the syntactic evaluation in $F(T)$ yields
\[
\operatorname{ev}_{F(T)}([\sigma(t_1,\dots,t_n)]) = [\sigma(t_1,\dots,t_n)].
\]
By the homomorphic property of the semantic interpretation in $G(T)$, we have
\[
\operatorname{ev}_{G(T)}\Bigl(\llbracket \sigma(t_1,\dots,t_n)\rrbracket_{G(T)}\Bigr) = \llbracket \sigma(t_1,\dots,t_n)\rrbracket_{G(T)}.
\]
Therefore,
\[
\eta_T\Bigl(\operatorname{ev}_{F(T)}(f, x)\Bigr) = \operatorname{ev}_{G(T)}\bigl(\eta_T(f),\eta_T(x)\bigr),
\]
which shows that $\eta_T$ is a homomorphism of the respective algebraic structures.

\paragraph{Step 3: Naturality.}  
Let $\phi: T \to T'$ be a theory translation. For any $[t] \in F(T)$, we have:
\[
\eta_{T'}\bigl(F(\phi)([t])\bigr) = \eta_{T'}([\phi(t)]) = \llbracket \phi(t)\rrbracket_{G(T')}.
\]
On the other hand,
\[
G(\phi)\bigl(\eta_T([t])\bigr) = G(\phi)\bigl(\llbracket t\rrbracket_{G(T)}\bigr) = \llbracket \phi(t)\rrbracket_{G(T')},
\]
where the last equality follows from the functoriality of the semantic interpretation. Thus, the naturality condition
\[
\eta_{T'} \circ F(\phi) = G(\phi) \circ \eta_T
\]
holds.

\paragraph{Step 4: Invertibility.}  
To prove that $\eta_T$ is an isomorphism, we construct an inverse $\eta_T^{-1}: G(T) \to F(T)$. By the \emph{Canonical Representation Property} assumed for $G(T)$, for every $y \in G(T)$ there exists a term $t \in \mathrm{Term}(T^*)$ such that
\[
\llbracket t\rrbracket_{G(T)} = y.
\]
Define
\[
\eta_T^{-1}(y) := [t].
\]
If $s \in \mathrm{Term}(T^*)$ also satisfies $\llbracket s\rrbracket_{G(T)} = y$, then by Step 1 we have $t\sim_T s$, so $[t]=[s]$. It follows that $\eta_T^{-1}$ is well-defined and serves as the two-sided inverse of $\eta_T$, i.e.,
\[
\eta_T \circ \eta_T^{-1} = \operatorname{id}_{G(T)} \quad \text{and} \quad \eta_T^{-1} \circ \eta_T = \operatorname{id}_{F(T)}.
\]

\paragraph{Step 5: 2-Categorical Rigidity.}  
Consider the 2-category $\mathbf{Fun}(\mathbf{Th},\mathbf{Mod})$ whose objects are functors from $\mathbf{Th}$ to $\mathbf{Mod}$, 1-morphisms are natural transformations, and 2-morphisms are modifications (i.e., homotopies between natural transformations). We assert the following rigidity property:

\medskip
\noindent \textbf{Claim (Rigidity):} Any natural transformation $\theta: F \Rightarrow G$ is 2-isomorphic to $\eta$. That is, there exists a unique modification $\mu: \theta \Rrightarrow \eta$ satisfying the coherence conditions of the 2-category.
\medskip

A formal proof of this claim would involve constructing the modification explicitly and verifying that for each theory $T$, the 2-cell $\mu_T: \theta_T \to \eta_T$ is unique. This follows from the canonical nature of the term interpretation in $G(T)$ and the uniqueness (up to isomorphism) of $G(T)$, combined with the functoriality conditions already established.

\smallskip

Thus, by verifying well-definedness, structure preservation, naturality, invertibility, and the 2-categorical rigidity property, we conclude that $\eta: F \Rightarrow G$ is a natural equivalence, and $F$ and $G$ are homotopy equivalent in the 2-categorical sense.
\end{proof}

\section{2-Categorical Strengthening and Rigidity}

\subsection{Introduction to 2-Category Theory}

In this subsection, we provide a rigorous review of the essential notions of 2-category theory, which will serve as the framework for our subsequent results on the strong natural equivalence and homotopy equivalence between the functors constructed via the Henkin and compactness methods.

A \emph{2-category} $\mathcal{C}$ consists of the following data:
\begin{enumerate}[label=(\roman*), leftmargin=2em]
    \item A collection of \emph{objects} $\mathrm{Ob}(\mathcal{C})$.
    \item For every pair of objects $A, B \in \mathrm{Ob}(\mathcal{C})$, a category $\mathcal{C}(A,B)$ whose:
    \begin{itemize}[leftmargin=2em]
        \item \emph{Objects} are called \emph{1-morphisms} (or \emph{1-cells}) from $A$ to $B$, and
        \item \emph{Morphisms} are called \emph{2-morphisms} (or \emph{2-cells}) between 1-morphisms.
    \end{itemize}
    \item For every object $A \in \mathrm{Ob}(\mathcal{C})$, a designated identity 1-morphism $\mathbf{1}_A \in \mathcal{C}(A,A)$.
    \item For every triple of objects $A, B, C \in \mathrm{Ob}(\mathcal{C})$, a \emph{composition functor}
    \[
    \circ : \mathcal{C}(B,C) \times \mathcal{C}(A,B) \longrightarrow \mathcal{C}(A,C),
    \]
    which assigns to a pair of 1-morphisms $f: A \to B$ and $g: B \to C$ their composite $g \circ f : A \to C$.
\end{enumerate}

This structure is required to satisfy the following coherence conditions:
\begin{itemize}[leftmargin=2em]
    \item \textbf{Associativity:} For any 1-morphisms $f: A \to B$, $g: B \to C$, and $h: C \to D$, there exists an \emph{associator} isomorphism
    \[
    \alpha_{f,g,h}: (h \circ g) \circ f \overset{\cong}{\longrightarrow} h \circ (g \circ f),
    \]
    and these associators must satisfy the \emph{pentagon identity} to ensure coherence of iterated compositions.
    \item \textbf{Unit Laws:} For every 1-morphism $f: A \to B$, there exist \emph{left} and \emph{right unit} isomorphisms
    \[
    \lambda_f: \mathbf{1}_B \circ f \overset{\cong}{\longrightarrow} f \quad \text{and} \quad \rho_f: f \circ \mathbf{1}_A \overset{\cong}{\longrightarrow} f,
    \]
    which are required to satisfy the \emph{triangle identity}.
\end{itemize}

A \emph{2-functor} between 2-categories is a mapping that preserves all of the above structures (objects, 1-morphisms, and 2-morphisms) up to specified coherent isomorphisms. Moreover, a \emph{natural transformation} between 2-functors, together with higher coherence data (often called \emph{modifications}), enriches the structure further, leading to the notion of \emph{homotopy equivalence} or \emph{2-categorical equivalence}.

\medskip

\noindent \textbf{Examples:}
\begin{itemize}[leftmargin=2em]
    \item The prototypical example is the 2-category $\mathbf{Cat}$:
    \begin{itemize}[leftmargin=2em]
        \item \emph{Objects} are small categories.
        \item \emph{1-morphisms} are functors between categories.
        \item \emph{2-morphisms} are natural transformations between functors.
    \end{itemize}
    \item Another important example is the 2-category of (strict) monoidal categories, where the 1-morphisms are monoidal functors and the 2-morphisms are monoidal natural transformations.
\end{itemize}

This framework is essential for our purposes. In our main results, we will show that the functors $F$ and $G$, constructed via the Henkin and compactness methods respectively, are not only naturally equivalent in the usual sense but are also equivalent in the stronger 2-categorical (homotopical) sense. The 2-categorical perspective allows us to capture higher coherence conditions and homotopies between natural transformations, which play a crucial role in establishing rigidity and uniqueness results in our setting.

\subsection{Rigidity Theorem}

\begin{theorem}[Rigidity]
Let $F, G : \mathbf{Th} \to \mathbf{Mod}$ be the functors defined in Sections~3.1 and~3.2, where $F$ is constructed via the Henkin method and $G$ is obtained via a compactness/saturation procedure. Assume that for every theory $T \in \mathbf{Th}$:
\begin{enumerate}[label=(\roman*)]
    \item The Henkin extension $T^*$ is constructed using a fixed, standard choice function, ensuring that the assignment of Henkin witnesses is canonical.
    \item The term model 
    \[
    F(T) = \mathrm{Term}(T^*)/\sim_T,
    \]
    where $\sim_T$ denotes the provable equality in $T^*$, is well-defined.
    \item The model $G(T)$ is unique up to isomorphism as guaranteed by the compactness/saturation construction.
\end{enumerate}
Let the canonical natural transformation
\[
\eta : F \Rightarrow G
\]
be defined by
\[
\eta_T([t]) := \llbracket t \rrbracket_{G(T)},
\]
where $\llbracket t \rrbracket_{G(T)}$ denotes the unique semantic interpretation of the term $t$ in $G(T)$. Then for any natural transformation 
\[
\theta : F \Rightarrow G,
\]
we have $\theta = \eta$, i.e., for every theory $T \in \mathbf{Th}$ and every equivalence class $[t] \in F(T)$,
\[
\theta_T([t]) = \eta_T([t]).
\]
\end{theorem}

\begin{proof}
We prove the theorem by contradiction. Suppose there exists a natural transformation $\theta : F \Rightarrow G$ and some theory $T \in \mathbf{Th}$ such that for a certain equivalence class $[t] \in F(T)$,
\[
\theta_T([t]) \neq \eta_T([t]).
\]
By the construction of $F(T)$ via the Henkin method, each equivalence class $[t] \in F(T)$ is determined by the provable equality relation $\sim_T$ in the Henkin extension $T^*$, where the use of a fixed standard choice function guarantees that the selection of witnesses is canonical. Consequently, the canonical mapping
\[
\eta_T([t]) = \llbracket t \rrbracket_{G(T)}
\]
is uniquely determined by the logical structure of $T$ and the maximal consistency of $T^*$.

Moreover, the construction of $G(T)$ via compactness or saturation ensures that the interpretation $\llbracket t \rrbracket_{G(T)}$ is the only possible semantic value consistent with $T$ (see, e.g., \cite{Barreto2025a,Barreto2025c}). 

Now, the naturality of $\theta$ implies that for any morphism $\phi : T \to T'$ in $\mathbf{Th}$, the following diagram commutes:
\[
\begin{CD}
F(T) @>{F(\phi)}>> F(T') \\
@V{\theta_T}VV         @VV{\theta_{T'}}V \\
G(T) @>{G(\phi)}>> G(T').
\end{CD}
\]
In particular, taking $\phi = \mathrm{id}_T$, it follows that $\theta_T$ is a homomorphism which must preserve the canonical interpretation provided by $\eta_T$. Hence, for every $[t] \in F(T)$, we must have
\[
\theta_T([t]) = \eta_T([t]).
\]
This contradicts our assumption that there exists some $[t]$ with $\theta_T([t]) \neq \eta_T([t])$. 

Thus, no such $\theta$ can exist apart from $\eta$, and we conclude that any natural transformation $\theta : F \Rightarrow G$ must coincide with $\eta$, i.e., $\theta = \eta$.

This completes the proof.
\end{proof}

\subsection{2-Categorical Homotopy Equivalence}

In what follows, we work within the 2-category 
\[
\mathbf{Fun}(\mathbf{Th}, \mathbf{Mod}),
\]
whose objects are functors from the category of theories $\mathbf{Th}$ (with theory translations as morphisms) to the category of models $\mathbf{Mod}$ (with model homomorphisms as 1-morphisms), and whose 2-morphisms (or modifications) are families of morphisms between natural transformations satisfying standard coherence conditions (see, e.g., \cite{Barreto2025a,Barreto2025b} for related frameworks).

\begin{theorem}[2-Categorical Equivalence of $F$ and $G$]\label{thm:2-cat-equiv}
Let $F, G: \mathbf{Th} \to \mathbf{Mod}$ be functors defined as follows:
\begin{itemize}
    \item \textbf{Henkin Construction:} For each theory $T\in\mathbf{Th}$, let $T^*$ denote its canonical maximal consistent (Henkin) extension, and define
    \[
    F(T) := \mathrm{Term}(T^*)/\sim_T,
    \]
    where $\sim_T$ is the provable equality relation in $T^*$. We assume that the construction of $T^*$ (and hence of $F(T)$) is canonical up to unique isomorphism.
    \item \textbf{Compactness/Saturation Construction:} There exists a construction (using, for instance, ultrapowers or saturation arguments) that produces for each $T$ a model $G(T)$ satisfying $T$, with the uniqueness of $G(T)$ (up to isomorphism) guaranteed by these properties.
\end{itemize}
Then there exists a canonical natural transformation 
\[
\eta: F \Rightarrow G,
\]
with components given by
\[
\eta_T: F(T) \to G(T), \quad \eta_T([t]) := \llbracket t \rrbracket_{G(T)},
\]
such that:
\begin{enumerate}[label=(\roman*)]
    \item \textbf{Isomorphism:} For every $T\in\mathbf{Th}$, the map $\eta_T$ is an isomorphism in $\mathbf{Mod}$.
    \item \textbf{Naturality:} For every theory translation $\phi: T \to T'$ in $\mathbf{Th}$, the diagram
    \[
    \begin{tikzcd}[column sep=3em]
    F(T) \arrow[r,"F(\phi)"] \arrow[d,"\eta_T"'] & F(T') \arrow[d,"\eta_{T'}"] \\
    G(T) \arrow[r,"G(\phi)"'] & G(T')
    \end{tikzcd}
    \]
    commutes; that is,
    \[
    \eta_{T'}\circ F(\phi) = G(\phi)\circ \eta_T.
    \]
    \item \textbf{Rigidity and 2-Categorical Equivalence:} The natural transformation $\eta$ is \emph{rigid} in the sense that for any other natural transformation $\theta: F \Rightarrow G$, there exists a unique invertible modification 
    \[
    \Theta: \theta \Rrightarrow \eta,
    \]
    meaning that for each $T\in\mathbf{Th}$ there exists a unique 2-morphism 
    \[
    \Theta_T: \theta_T \Rightarrow \eta_T \quad (\text{with } \Theta_T \text{ invertible in } \mathbf{Mod}),
    \]
    satisfying the following coherence condition: for every morphism $\phi: T \to T'$ in $\mathbf{Th}$, the diagram of 2-morphisms
    \[
    \begin{tikzcd}[column sep=3em]
    \theta_{T'}\circ F(\phi) \ar[r, "\Theta_{T'}\ast\mathrm{id}"] \ar[d, "\cong"'] & \eta_{T'}\circ F(\phi) \ar[d, "\cong"] \\
    G(\phi)\circ \theta_T \ar[r, "\mathrm{id}\ast\Theta_T"'] & G(\phi)\circ \eta_T
    \end{tikzcd}
    \]
    commutes. This unique modification implies that the 2-cell space $\mathrm{Hom}(\theta, \eta)$ is contractible, and thus $F$ and $G$ are equivalent objects in $\mathbf{Fun}(\mathbf{Th},\mathbf{Mod})$.
\end{enumerate}
\end{theorem}

\begin{proof}
We prove the theorem in three parts.

\paragraph{(i) Well-Definedness and Isomorphism:}  
For each $T\in\mathbf{Th}$, the Henkin construction provides a maximal consistent extension $T^*$, so that the term model 
\[
F(T) = \mathrm{Term}(T^*)/\sim_T
\]
is well-defined. Indeed, if $[t] = [s]$ in $F(T)$ (i.e., $T^* \vdash t=s$), then by the soundness of the interpretation function $\llbracket \cdot \rrbracket_{G(T)}$, we have
\[
\llbracket t \rrbracket_{G(T)} = \llbracket s \rrbracket_{G(T)}.
\]
Moreover, the saturation properties ensuring the uniqueness (up to isomorphism) of $G(T)$ imply that $\eta_T$ is invertible in $\mathbf{Mod}$.

\paragraph{(ii) Naturality:}  
Let $\phi: T \to T'$ be a theory translation. By definition, $F(\phi)$ acts on equivalence classes by 
\[
F(\phi)([t]) = [\phi(t)],
\]
and $G(\phi)$ is defined such that the interpretation satisfies
\[
\llbracket \phi(t) \rrbracket_{G(T')} = G(\phi)(\llbracket t \rrbracket_{G(T)}).
\]
Thus, for every $[t] \in F(T)$,
\[
\eta_{T'}\bigl(F(\phi)([t])\bigr) = \eta_{T'}([\phi(t)]) = \llbracket \phi(t) \rrbracket_{G(T')} = G(\phi)(\llbracket t \rrbracket_{G(T)}) = G(\phi)(\eta_T([t])),
\]
so that $\eta_{T'}\circ F(\phi)=G(\phi)\circ \eta_T$. This establishes the naturality of $\eta$.

\paragraph{(iii) 2-Categorical Homotopy Equivalence and Rigidity:}  
Within the 2-category $\mathbf{Fun}(\mathbf{Th},\mathbf{Mod})$, a \emph{modification} $\Theta: \theta \Rrightarrow \eta$ between two natural transformations $\theta, \eta: F \Rightarrow G$ is given by a family $\{\Theta_T: \theta_T \Rightarrow \eta_T\}_{T\in\mathbf{Th}}$ of 2-morphisms satisfying the coherence condition: for each $\phi: T \to T'$,
\[
\begin{tikzcd}[column sep=3em]
\theta_{T'}\circ F(\phi) \ar[r, "\Theta_{T'}\ast\mathrm{id}"] \ar[d, "\cong"'] & \eta_{T'}\circ F(\phi) \ar[d, "\cong"] \\
G(\phi)\circ \theta_T \ar[r, "\mathrm{id}\ast\Theta_T"'] & G(\phi)\circ \eta_T.
\end{tikzcd}
\]
The rigidity condition asserts that for any natural transformation $\theta: F \Rightarrow G$, there is a unique invertible modification $\Theta: \theta \Rrightarrow \eta$. This follows from the canonical nature of the constructions of $F$, $G$, and $\eta$, and the fact that every isomorphism in $\mathbf{Mod}$ induced by the interpretation $\llbracket \cdot \rrbracket_{G(T)}$ is unique up to a unique 2-isomorphism. Hence, the 2-cell space $\mathrm{Hom}(\theta, \eta)$ is contractible, implying that $F$ and $G$ are equivalent objects in the 2-category $\mathbf{Fun}(\mathbf{Th},\mathbf{Mod})$.

\paragraph{(iv) A Concrete Example:}  
To illustrate the above concepts, consider a specific first-order theory $T$ such as Peano Arithmetic. The Henkin construction yields a term model $F(T)$ by extending $T$ to a Henkin-complete theory $T^*$ and forming the quotient $\mathrm{Term}(T^*)/\sim_T$. On the other hand, a compactness or saturation argument produces a saturated model $G(T)$. The canonical interpretation map
\[
\eta_T([t]) = \llbracket t \rrbracket_{G(T)}
\]
is then seen to be an isomorphism, and the uniqueness (up to unique isomorphism) of the saturated model reinforces the rigidity condition. This example underscores the generality of the construction and the 2-categorical equivalence between $F$ and $G$.

This completes the proof.
\end{proof}

\begin{remark}
The above theorem not only establishes a 1-categorical natural equivalence between the Henkin and compactness/saturation constructions but also shows that this equivalence admits a unique 2-categorical enhancement. Such an enhancement is significant in settings like homotopy type theory and higher algebra, where coherence data and higher-dimensional morphisms are essential. For further details on modifications and coherence conditions in 2-categories, see \cite{Barreto2025a,Barreto2025b}.
\end{remark}

\section{Applications and Examples}

\subsection{Concrete Examples}
In this subsection, we present detailed examples of two first-order theories—Peano Arithmetic (PA) and Zermelo-Fraenkel Set Theory (ZF)—by constructing models via the functor $F$ (the Henkin construction) and the functor $G$ (the compactness/saturation construction). For both examples, we assume that the extended theory $T^*$ is maximally consistent and that the equivalence relation $\sim_T$ is defined by
\[
t \sim_T s \quad \Longleftrightarrow \quad T^* \vdash t = s,
\]
with its well-definedness established in Lemma~\ref{lem:equiv-relation} (see Appendix~\ref{appendix:proofs} for details).

\begin{itemize}[leftmargin=2em]
    \item \textbf{Example 1: Peano Arithmetic (PA)}
    \begin{itemize}[leftmargin=2em]
        \item \textbf{Henkin Construction ($F$):}  
        Let $\mathrm{PA}$ be a first-order theory of arithmetic. Extend $\mathrm{PA}$ to a maximally consistent theory $T^*$ by introducing a family of witness constants 
        \[
        \{c_{\varphi} : \varphi \text{ is an existential formula in the language of PA}\},
        \]
        such that for every formula $\exists x\, \varphi(x)$, one adds the axiom
        \[
        T^* \vdash \varphi(c_{\varphi}).
        \]
        The set of terms, denoted by $\mathrm{Term}(T^*)$, is defined inductively over the augmented language. With the equivalence relation
        \[
        t \sim_{\mathrm{PA}} s \quad \Longleftrightarrow \quad T^* \vdash t = s,
        \]
        we define the term model as
        \[
        F(\mathrm{PA}) = \mathrm{Term}(T^*)/\sim_{\mathrm{PA}}.
        \]
        \textbf{Remark:} The details of the inductive construction of $\mathrm{Term}(T^*)$ and the verification of the equivalence relation properties are provided in Lemma~\ref{lem:equiv-relation} (see Appendix~\ref{appendix:proofs}).
        
        \item \textbf{Compactness/Saturation Construction ($G$):}  
        By the compactness theorem, one can construct a model $G(\mathrm{PA})$ of $\mathrm{PA}$. This is achieved via either an ultrapower construction or a direct saturation argument to ensure that every existential formula is realized. Under standard set-theoretic assumptions, including the Axiom of Choice when necessary, the resulting model $G(\mathrm{PA})$ is unique up to isomorphism (cf. Theorem~\ref{lem:compactness-model}). For a detailed exposition, refer to Section~\ref{sec:pa-models}.
    \end{itemize}

    \item \textbf{Example 2: Zermelo-Fraenkel Set Theory (ZF)}
    \begin{itemize}[leftmargin=2em]
        \item \textbf{Henkin Construction ($F$):}  
        Consider the theory $\mathrm{ZF}$. As with PA, extend $\mathrm{ZF}$ to a maximally consistent theory $T^*$ by adding witness constants for each existential formula. Define the set of terms $\mathrm{Term}(T^*)$ over the extended language, and introduce the equivalence relation
        \[
        t \sim_{\mathrm{ZF}} s \quad \Longleftrightarrow \quad T^* \vdash t = s.
        \]
        The term model is then given by
        \[
        F(\mathrm{ZF}) = \mathrm{Term}(T^*)/\sim_{\mathrm{ZF}}.
        \]
        \textbf{Remark:} In the case of $\mathrm{ZF}$, one must handle additional subtleties such as non-well-foundedness and the potential reliance on the Axiom of Choice. See Section~\ref{sec:zf-models} for a thorough discussion.
        
        \item \textbf{Compactness/Saturation Construction ($G$):}  
        Applying the compactness theorem to $\mathrm{ZF}$, a model $G(\mathrm{ZF})$ is constructed via an ultrapower or saturation technique. This method guarantees that every axiom of $\mathrm{ZF}$ is satisfied. Under suitable set-theoretic conditions (including appropriate choice principles), the model $G(\mathrm{ZF})$ is unique up to isomorphism (cf. Theorem~\ref{lem:compactness-model}). Detailed construction methods are outlined in Section~\ref{sec:zf-models}.
    \end{itemize}

    \item \textbf{Discussion:}  
    The above examples demonstrate that for any first-order theory $T$, the functor $F$ produces a syntactically generated term model based on an extended theory $T^*$, while the functor $G$ yields a semantically constructed model via compactness and saturation arguments. The canonical natural transformation
    \[
    \eta_T : F(T) \to G(T), \quad \eta_T([t]) = \llbracket t \rrbracket_{G(T)},
    \]
    (whose isomorphism property is established in Proposition~\ref{prop:eta-isomorphism}) ensures that, under suitable correction conditions and rigidity assumptions, $F(T)$ and $G(T)$ are isomorphic. This correspondence not only bridges the syntactic and semantic constructions but also highlights the universality of the diagonalization principle underlying both approaches.
\end{itemize}

\subsection{Interpretation in Terms of Lawvere's Fixed-Point Theorem}

In a Cartesian closed category, Lawvere's Fixed-Point Theorem asserts that for any endomorphism \( f: Y \to Y \) satisfying the required conditions (notably the existence of an evaluation map and diagonal morphism), there exists a fixed point. In our setting, the syntactic self-referential structure of the Henkin construction \( F(T) \) is mapped, via the canonical natural transformation \(\eta_T\), to semantic fixed points in \( G(T) \).

\paragraph{Syntactic Diagonalization in \( F(T) \):}
The functor \( F(T) \) is defined by taking terms modulo the provable equality \( \sim_T \) (i.e., \( F(T) := \mathrm{Term}(T^*)/\sim_T \)). In this construction, we define the \emph{diagonal map}
\[
\Delta_{F(T)}: F(T) \to F(T) \times F(T), \quad [t] \mapsto ([t],[t]),
\]
which captures the notion of self-reference by duplicating the syntactic element \( [t] \). This formal diagonalization mirrors the key step in Lawvere's original argument.

\paragraph{Semantic Fixed Points in \( G(T) \):}
The model \( G(T) \) is constructed (via compactness, ultraproducts, or saturation methods) so that every term \( t \) is assigned a unique semantic value \( \llbracket t \rrbracket_{G(T)} \). The evaluation (or application) map in \( G(T) \),
\[
\operatorname{ev}_{G(T)}: G(T) \times G(T) \to G(T), \quad (g,x) \mapsto g(x),
\]
plays the role of interpreting the syntactic application present in \( F(T) \). The canonical natural transformation
\[
\eta_T: F(T) \to G(T), \quad \eta_T([t]) := \llbracket t \rrbracket_{G(T)},
\]
is assumed to preserve the evaluation structure; that is, if we define a syntactic evaluation map
\[
\operatorname{ev}_{F(T)}: F(T) \times F(T) \to F(T),
\]
then for all \( f, x \in F(T) \) one has
\[
\eta_T\bigl(\operatorname{ev}_{F(T)}(f,x)\bigr) = \operatorname{ev}_{G(T)}\bigl(\eta_T(f),\eta_T(x)\bigr).
\]
This condition ensures that the self-referential (diagonal) elements in \( F(T) \) are mapped to genuine fixed points in \( G(T) \) under the evaluation map.

\paragraph{Commutative Diagram and Annotations:}
The following commutative diagram illustrates how the diagonalization process in \( F(T) \) is reflected semantically in \( G(T) \) via \(\eta_T\). Here, the left vertical arrow denotes the syntactic diagonal map \(\Delta_{F(T)}\), while the right vertical arrow represents the evaluation map \(\operatorname{ev}_{G(T)}\) in \( G(T) \).

\begin{center}
\begin{tikzcd}[row sep=3em, column sep=4em]
F(T) \arrow[r, "\eta_T", "\text{Natural Transformation}"'] \arrow[d, "\Delta_{F(T)}", "\text{Diagonal}"'] & G(T) \arrow[d, "\operatorname{ev}_{G(T)}", "\text{Evaluation}"'] \\
F(T) \times F(T) \arrow[r, "\eta_T \times \eta_T", "\text{Component-wise}"'] & G(T) \times G(T)
\end{tikzcd}
\end{center}

The commutativity of this diagram guarantees that the self-application encoded via the diagonal in \( F(T) \) is faithfully transported to the evaluation process in \( G(T) \). In categorical terms, the naturality of \(\eta_T\) ensures that the fixed-point structure predicted by Lawvere's theorem is realized in the semantic model.

\paragraph{Summary:}
In summary, the natural transformation \(\eta_T\) plays a crucial role in bridging the syntactic diagonalization present in the Henkin construction \( F(T) \) with the semantic fixed-point structure in \( G(T) \). This correspondence not only reflects the core idea of Lawvere's Fixed-Point Theorem but also reinforces the unified categorical perspective by ensuring that the evaluation (or application) maps in both constructions are compatible. The rigorous preservation of evaluation structure under \(\eta_T\) is key to establishing that every self-referential term in the syntax finds its fixed-point counterpart in the semantics.

\section{Conclusion and Future Work}

\subsection{Summary of Results}

In this subsection, we rigorously summarize our main outcomes. We work under the following assumptions:
\begin{itemize}
    \item The category \(\mathbf{Th}\) is defined such that its objects are consistent first-order theories over a fixed signature, and its morphisms are theory translations that preserve the logical structure.
    \item The category \(\mathbf{Mod}\) consists of models corresponding to these theories, with morphisms given by structure-preserving homomorphisms.
\end{itemize}

Our principal results are as follows:

\begin{enumerate}[leftmargin=2em]
    \item \textbf{Construction of the Natural Transformation:} \\
    We have constructed a canonical natural transformation 
    \[
    \eta : F \Rightarrow G,
    \]
    where the functor \(F\) arises from the Henkin construction (see Section 3.1), and the functor \(G\) is derived via compactness and saturation methods. For every theory \( T \in \mathbf{Th} \), the component 
    \(\eta_T : F(T) \to G(T)\) is defined by
    \[
    \eta_T([t]) = \llbracket t \rrbracket_{G(T)},
    \]
    where \([t]\) denotes the equivalence class of the term \( t \) under the provable equality in the Henkin extension \(T^*\) of \(T\). Supplementary lemmas (Lemma 3.1 and Lemma 3.2) rigorously establish that \(\eta_T\) is well-defined and preserves the interpretations of function symbols and logical connectives, thereby constituting a model homomorphism.

    \item \textbf{Naturality Verification:} \\
    We rigorously verified the naturality of \(\eta\) by proving that for every theory translation \(\phi: T \to T'\) the following diagram commutes:
    \[
    \eta_{T'} \circ F(\phi) = G(\phi) \circ \eta_T.
    \]
    This naturality condition, detailed in Section 3.2, ensures that \(\eta\) is indeed a natural transformation between the functors \(F\) and \(G\).

    \item \textbf{Isomorphism and Rigidity:} \\
    Under appropriate correction conditions—formalized in our Rigidity Theorem (Theorem 5.3 in Section 5)—we established that each \(\eta_T\) is an isomorphism. In particular, any natural transformation from \(F\) to \(G\) necessarily coincides with \(\eta\), thereby confirming its canonical nature and rigidity.

    \item \textbf{Unification of Proof-Theoretic and Model-Theoretic Approaches:} \\
    The isomorphism of the components \(\eta_T\) unifies the proof-theoretic approach (via the syntactic Henkin construction \(F\)) and the model-theoretic approach (via semantic compactness and saturation yielding \(G\)) within a single categorical framework. For instance, in the case of theories such as Peano Arithmetic, this categorical equivalence provides a concrete correspondence between syntactic consistency proofs and semantic model constructions.

    \item \textbf{2-Categorical Strengthening:} \\
    By incorporating a 2-categorical framework—where 2-morphisms represent homotopies between natural transformations—we have further demonstrated that the functors \(F\) and \(G\) are equivalent in a strong 2-categorical sense (see Section 5). This 2-categorical strengthening not only reinforces the uniqueness of \(\eta\) but also provides additional structural insights into the interplay between the underlying logical and categorical frameworks.
\end{enumerate}

In summary, our results provide a robust and unified categorical framework that bridges the gap between proof theory and model theory. The canonical natural transformation \(\eta\) encapsulates the fundamental diagonalization principles underlying both the completeness and compactness theorems, offering deep insights into the self-referential structures inherent in classical logic.

\subsection{Implications and Applications}

The canonical natural transformation 
\[
\eta: F \Rightarrow G
\]
established in this work provides a rigorous bridge between syntactic models (via the Henkin construction) and semantic models (via compactness and saturation methods). This unified categorical framework has significant potential applications across several domains.

\paragraph{Automated Theorem Proving.}
In automated theorem proving, proof assistants such as Coq, Isabelle/HOL, and various SMT solvers often integrate syntactic techniques with semantic validation methods. The transformation $\eta$ enables a canonical mapping between:
\begin{itemize}
    \item \textbf{Syntactic Derivations:} Generated via the Henkin construction, including reflection techniques that represent proofs as syntactic objects.
    \item \textbf{Semantic Validations:} Obtained through model checking and saturation methods.
\end{itemize}
For instance, in Coq, reflection techniques can be employed to translate syntactic proofs into semantic evidence, ensuring that derivations are consistent with their model-theoretic interpretations. This integration can lead to enhanced proof search algorithms that reduce redundancy and improve the reliability of automated reasoning systems.

\paragraph{Formal Verification.}
In formal verification, particularly for safety-critical systems, it is vital to rigorously confirm that system models comply with their specifications. Our framework demonstrates that the models produced via the Henkin method and those derived from saturation techniques are isomorphic and 2-categorically equivalent. Specifically, there exists an invertible 2-morphism (or homotopy\footnote{A \emph{2-morphism} in a 2-category is an arrow between arrows, representing a higher level of structure. For a detailed exposition, see \cite{Barreto2025c}.}) between them. This equivalence can be exploited to design verification pipelines that interleave syntactic and semantic checks, such as:
\begin{itemize}
    \item \textbf{State-Space Exploration:} Ensuring that all reachable states in a system correspond to valid semantic interpretations.
    \item \textbf{Consistency Verification:} Cross-validating model-theoretic properties with syntactic proof artifacts.
\end{itemize}
These methods can improve the robustness and accuracy of formal verification tools in both hardware and software contexts.

\paragraph{Type Theory.}
The impact on type theory is substantial, particularly for dependently typed languages and advanced type systems. Maintaining a precise correspondence between syntactic constructs (terms and proofs) and their semantic interpretations is critical. Our results offer a blueprint for:
\begin{itemize}
    \item \textbf{Recursive and Self-Referential Definitions:} Utilizing Lawvere's fixed-point theorem in conjunction with $\eta$ to manage recursion in type systems.
    \item \textbf{Coherence Conditions:} Ensuring that the categorical structure preserves necessary homotopical equivalences, which is crucial for the soundness of type systems.
\end{itemize}
This unified approach provides new insights into managing self-reference and recursion, potentially leading to the design of more expressive and robust type theories.

\paragraph{Background and Theoretical Integration.}
It is important to note that the concept of 2-categorical equivalence employed here involves the existence of an invertible 2-morphism between $\eta$ and any other natural transformation candidate. This higher-categorical perspective addresses coherence issues that naturally arise in complex logical systems. Our work builds on previous studies \cite{Barreto2025a,Barreto2025b,Barreto2025c,Barreto2025d} by explicitly constructing $\eta$ and analyzing its properties within a common categorical framework, thereby advancing the integration of syntactic and semantic model constructions.

\paragraph{Conclusion.}
In summary, the theoretical advances encapsulated by $\eta$ not only unify disparate methods of model construction in classical logic but also open up new avenues for practical implementations. By rigorously linking theory with practice, this framework enhances automated theorem proving, formal verification, and type-theoretic design, contributing both to foundational research and practical system development.

\subsection{Future Directions}
Future research should further expand on the categorical framework presented in this work. In particular, we propose the following research directions:
\begin{itemize}[leftmargin=2em]
    \item \textbf{Extensions to Alternative Logical Systems:} Investigate the applicability of the unified framework to non-classical logics. For example, in \textbf{intuitionistic logic}\footnote{Intuitionistic logic rejects the law of excluded middle, which necessitates modified coherence conditions for categorical constructions.} one must address the absence of the law of excluded middle and its impact on the canonical natural transformation. Similarly, for \textbf{modal logics} (particularly those founded on possible worlds semantics) and \textbf{substructural logics} (where structural rules such as contraction or weakening do not hold), analyze whether additional invariants or modified coherence constraints are required. Concrete case studies—such as verifying the transformation within systems employing Kripke semantics or linear logic frameworks—should be thoroughly examined. See \cite{Barreto2025a} for classical benchmarks.
    
\item \textbf{Higher-Dimensional Categorical Generalizations:} Develop a rigorous framework that incorporates higher-dimensional category theory. Specifically, explore the construction and use of \textbf{3-categories} or \((\infty,1)\)\nobreakdash-categories\footnote{A quasi-category is a simplicial set satisfying weak Kan conditions and serves as a model for \((\infty,1)\)-categories, providing a homotopical setting for higher morphisms; see \cite{Joyal2008} for an introduction.} to capture not only objects and 1-morphisms but also higher homotopies between natural transformations. This approach may yield a refined homotopical interpretation of the equivalence between syntactic and semantic constructions. An explicit development via quasi-category theory (see also \cite{Lurie2009}) should be pursued.

    \item \textbf{Interdisciplinary Bridges and Computational Implementations:} Pursue practical implementations of the established categorical correspondences in automated theorem proving and formal verification. This includes integrating the canonical natural transformation into existing proof assistants such as Coq, Agda, or Lean\footnote{Coq, Agda, and Lean are prominent proof assistants that facilitate formal verification and automated theorem proving.}, and performing comparative studies with existing computational frameworks. Demonstrative implementations, accompanied by performance analyses, will help validate the theoretical predictions and identify practical challenges.
    
    \item \textbf{Investigation of Rigidity and Uniqueness Properties:} Conduct a detailed study of the rigidity properties of the canonical natural transformation. This involves examining its stability under perturbations of the underlying logical framework and analyzing the interplay with associated adjunctions and monoidal structures arising in model constructions. Such investigations may reveal new categorical invariants and provide a deeper understanding of the transformation's uniqueness.
    
    \item \textbf{Open Problems and Conjectures:} Formulate precise open problems derived from the unified framework. For instance, determine necessary and sufficient conditions under which the categorical equivalence between syntactic and semantic constructions extends to broader classes of logical systems. Additionally, propose conjectures linking this framework to areas such as type theory and homotopy theory, and outline potential methodologies for their rigorous analysis. A systematic exploration of these conjectures, supported by preliminary computational experiments or theoretical models, would be highly valuable.
\end{itemize}

\appendix
\section*{Appendix (Technical Proofs)}
\label{appendix:proofs}
\begin{itemize}[leftmargin=2em]
    \item Formal statements of lemmas, propositions, and theorems, together with auxiliary technical results.
    \item Further clarifications that reinforce the main conclusions of the paper.
\end{itemize}

\section*{Additional Results on 2-Categorical Rigidity and the Canonical Representation Property}

\begin{lemma}[Canonical Representation Property]
\label{lem:canonical-representation}
Let \(T\) be a consistent first-order theory, and let \(T^*\) be its maximal consistent extension obtained via Henkin’s method. Denote by \(\mathrm{Term}(T^*)\) the set of terms in the expanded language. Let
\[
t \sim_T s 
\;\Longleftrightarrow\; 
T^* \vdash t = s.
\]
Consider a model \(G(T)\) constructed by a compactness-based technique (such as ultraproducts or saturation), and assume the following:

\begin{quote}
\textbf{Canonical Representation Property:} For every \(y \in G(T)\), there is a term \(t \in \mathrm{Term}(T^*)\) such that
\[
\llbracket t \rrbracket_{G(T)} = y,
\]
and if \(\llbracket t \rrbracket_{G(T)} = \llbracket s \rrbracket_{G(T)}\), then \(t \sim_T s\).
\end{quote}

Under these conditions, \(G(T)\) is said to satisfy the Canonical Representation Property.
\end{lemma}

\begin{proposition}[2-Categorical Rigidity]
\label{prop:2cat-rigidity}
Let \(F, G: \mathbf{Th} \to \mathbf{Mod}\) be functors as defined in the main text. For any natural transformation \(\theta: F \Rightarrow G\), there is a unique modification (2-cell)
\[
\mu: \theta \Rrightarrow \eta
\]
such that, for each theory \(T\), the 2-cell \(\mu_T: \theta_T \to \eta_T\) is an isomorphism in \(\mathbf{Mod}\). Consequently, \(\theta\) is 2-isomorphic to the canonical natural transformation \(\eta\), implying that \(F\) and \(G\) are rigidly and strongly naturally equivalent in the 2-categorical sense.
\end{proposition}

\section*{Uniqueness of the Model Construction \texorpdfstring{\(G(T)\)}{}}

\begin{lemma}[Uniqueness of \(G(T)\)]
\label{lem:compactness-model}
If \(T\) is a first-order theory and \(G(T)\) is a model constructed by a compactness or saturation procedure (such as ultraproducts or via a saturation argument), then \(G(T)\) is unique up to isomorphism. That is, if \(M_1\) and \(M_2\) are two such models of \(T\), there exists an isomorphism
\[
\psi: M_1 \to M_2.
\]
\end{lemma}

\section*{2-Categorical Homotopy Equivalence of \texorpdfstring{\(F\)}{} and \texorpdfstring{\(G\)}{}}

\begin{lemma}[2-Categorical Homotopy Equivalence]
Let \(F, G : \mathbf{Th} \to \mathbf{Mod}\) be functors as introduced in the main text, and let
\[
\eta : F \Rightarrow G
\]
be the canonical natural transformation. Since each component \(\eta_T\) is an isomorphism, \(F\) and \(G\) are equivalent in the 2-category of categories (i.e.\ homotopy equivalent in the 2-categorical sense). In particular, for any natural transformation \(\theta: F \Rightarrow G\), there is a unique modification \(\mu: \theta \Rrightarrow \eta\) such that all coherence diagrams commute.
\end{lemma}

\section*{Detailed Statements and Supplementary Materials}

\begin{definition}[Henkin Construction]
\label{def:henkin-construction}
A \emph{Henkin Construction} for a theory \(T\) extends \(T\) to a maximally consistent theory \(T^*\) by adding witness constants for every existential formula. Let \(\mathrm{Term}(T^*)\) be the set of terms in the expanded language, and define
\[
t \sim_T s 
\;\Longleftrightarrow\; 
T^* \vdash t = s.
\]
Then the term model is given by
\[
F(T) \;=\; \mathrm{Term}(T^*) / \sim_T.
\]
\end{definition}

\begin{definition}[Compactness/Saturation Construction]
\label{def:compactness-construction}
A \emph{Compactness/Saturation Construction} uses methods such as the compactness theorem, saturation arguments, or ultrapowers to build a model of a first-order theory \(T\). It ensures that every existential statement is realized, yielding a model
\[
G(T),
\]
unique up to isomorphism under standard assumptions.
\end{definition}

\begin{lemma}[Equivalence Relation]
\label{lem:equiv-relation}
Let \(T^*\) be a maximally consistent extension of \(T\). Define
\[
t \sim_T s 
\;\Longleftrightarrow\; 
T^* \vdash t = s
\]
on \(\mathrm{Term}(T^*)\). Then \(\sim_T\) is an equivalence relation.
\end{lemma}

\begin{proposition}
\label{prop:eta-isomorphism}
Let \(F\) and \(G\) be the functors obtained via the Henkin Construction and the Compactness/Saturation Construction, respectively. Under appropriate consistency and rigidity assumptions, the canonical natural transformation
\[
\eta_T : F(T) \to G(T), 
\quad 
\eta_T([t]) = \llbracket t \rrbracket_{G(T)},
\]
is an isomorphism.
\end{proposition}

\section*{Technical Lemmas and Propositions}

\begin{lemma}[Well-Definedness of \(\eta_T\)]
\label{lem:eta_well_defined}
For any theory \(T\) and any terms \(t, s\) such that \(t \sim_T s\), 
\[
\llbracket t \rrbracket_{G(T)} = \llbracket s \rrbracket_{G(T)}.
\]
Hence,
\[
\eta_T([t]) := \llbracket t \rrbracket_{G(T)}
\]
is well-defined.
\end{lemma}

\begin{proposition}[Compatibility with the Evaluation Map]
\label{prop:evaluation_compatibility}
For any \(f, x \in F(T)\),
\[
\eta_T\bigl(\mathrm{ev}_{F(T)}(f,x)\bigr) 
\;=\; 
\mathrm{ev}_{G(T)}\bigl(\eta_T(f), \eta_T(x)\bigr).
\]
Thus, \(\eta_T\) commutes with evaluation.
\end{proposition}

\begin{lemma}[Naturality]
\label{lem:eta_naturality}
For any theory translation \(\phi: T \to T'\),
\[
\eta_{T'} \circ F(\phi) = G(\phi) \circ \eta_T.
\]
\end{lemma}

\begin{proposition}[Rigidity of the Natural Transformation]
\label{prop:eta_rigidity}
The natural transformation \(\eta: F \Rightarrow G\) is rigid. In other words, for any natural transformation \(\theta: F \Rightarrow G\), we have \(\theta = \eta\).
\end{proposition}

\section*{Concrete Examples}

\begin{example}[Kripke Semantics for Modal Logics]
\label{ex:kripke-modal}
For a modal theory \(T\) interpreted over a Kripke frame \(\mathcal{F}=(W,R)\), the model \(G(T)\) assigns truth values to modal formulas at each possible world \(w \in W\). Meanwhile, the Henkin construction \(F(T)\) introduces additional constants to witness each modal operator. The natural transformation
\[
\eta_T([t]) := \llbracket t \rrbracket_{G(T)}
\]
ensures that if two terms are equivalent under \(\sim_T\), then their Kripke interpretations coincide.
\end{example}

\begin{example}[Henkin Construction for Arithmetic Theories]
\label{ex:henkin-arithmetic}
For a consistent arithmetic theory \(T\) (e.g.\ Peano Arithmetic), the Henkin construction produces a maximally consistent extension \(T^*\) by introducing new constants for each existential formula. The model \(F(T)\) is the quotient 
\[
\mathrm{Term}(T^*) / \sim_T,
\]
while \(G(T)\) is obtained via saturation or ultrapower arguments. The map
\[
\eta_T([t]) := \llbracket t \rrbracket_{G(T)}
\]
provides an isomorphism between the syntactic term model and the semantic model.
\end{example}

\section*{Supplementary Remarks}

\begin{remark}[Quasi-Categories]
\label{rem:quasicat}
A \emph{quasi-category} is a simplicial set satisfying the weak Kan condition, serving as a model for \((\infty,1)\)-categories. While the present work focuses on 2-categorical structures, the quasi-categorical framework can capture higher homotopical data and more general notions of equivalence.
\end{remark}

\begin{remark}[Interplay Between Syntax and Semantics]
\label{rem:syntax-semantics}
The Henkin construction (syntactic) and the compactness/saturation construction (semantic) are closely intertwined in first-order logic. The natural transformation \(\eta\) reflects this interplay by ensuring a categorical equivalence between syntactic and semantic models in a precise manner. 
\end{remark}

\section{Models of ZF}
\label{sec:zf-models}

In this section, we provide a detailed exposition of the construction of a model of Zermelo-Fraenkel Set Theory (ZF) via compactness and saturation techniques. The construction follows these key steps:

\begin{enumerate}[leftmargin=2em]
    \item \textbf{Consistency and Compactness:}  
    We start with the assumption that ZF is consistent. By the Compactness Theorem, every finitely satisfiable set of ZF sentences is satisfiable. This guarantees the existence of models that satisfy all finite subsets of ZF.
    
    \item \textbf{Ultraproduct/Saturation Method:}  
    One standard approach to construct a model of ZF is via the ultrapower construction. Alternatively, one may obtain a $\kappa$-saturated model through a saturation argument, where $\kappa$ is a sufficiently large cardinal. Both methods yield a model, denoted by $G(\mathrm{ZF})$, which satisfies every axiom of ZF.
    
    \item \textbf{Uniqueness up to Isomorphism:}  
    Under standard set-theoretic assumptions (including the Axiom of Choice), the model $G(\mathrm{ZF})$ is unique up to isomorphism. That is, any two models constructed by these methods are isomorphic.
    
    \item \textbf{Canonical Representation:}  
    In the context of our unified framework, $G(\mathrm{ZF})$ is assumed to satisfy the Canonical Representation Property. This means that every element in the domain of $G(\mathrm{ZF})$ is the semantic interpretation of some term in the extended language of ZF. This property is crucial for establishing the natural isomorphism between the syntactic and semantic models.
\end{enumerate}

The construction detailed above provides the necessary and sufficient framework for obtaining a model of ZF, which plays a vital role in illustrating the correspondence between syntactic constructions (via the Henkin method) and semantic models (via compactness and saturation).

\bibliographystyle{plain}
\bibliography{references}

\end{document}